\documentclass[english]{amsart}
\usepackage{babel}
\usepackage[]{geometry}
\usepackage{caption}

\usepackage{caption}
\usepackage{subcaption}

\usepackage{amsmath,amssymb,amsfonts,latexsym,cancel}
\usepackage{rawfonts}
\usepackage{pictexwd}
\usepackage{tikz}

\usepackage{tikz-cd}

\usepackage{pgfplots}

\pgfplotsset{compat=newest}
\pgfplotsset{plot coordinates/math parser=false}
\newlength\figureheight
\newlength\figurewidth

\usepackage{color}
\usepackage{epsfig}

\newcommand{\R}{\mathbb{R}}

\def\1{\raisebox{2pt}{\rm{$\chi$}}}

\newcommand{\Lip}{\operatorname{Lip}}

\theoremstyle{plain}
\newtheorem{definition}{Definition}[section]
\newtheorem{proposition}{Proposition}[section]
\newtheorem{property}{Property}[section]
\newtheorem{theorem}{Theorem}[section]
\newtheorem{corollary}{Corollary}[section]
\newtheorem{lemma}{Lemma}[section]
\newtheorem{remark}{Remark}[section]

\theoremstyle{definition}
\newtheorem{example}{Example}

\theoremstyle{remark}

\numberwithin{equation}{section}

\begin{document}

\title[The inverse problem for Hamilton-Jacobi equations]{The inverse problem for Hamilton-Jacobi equations and semiconcave envelopes}

\begin{abstract}
We study the inverse problem, or inverse design problem, for a time-evolution Hamilton-Jacobi equation.
More precisely, given a target function $u_T$ and a time horizon $T>0$, we aim to construct all the initial conditions for which the viscosity solution coincides with $u_T$ at time $T$.
As it is common in this kind of nonlinear equations, the target might not be reachable.
We first study the existence of at least one initial condition leading the system to the given target.
The natural candidate, which indeed allows determining the reachability of $u_T$, 
is the one obtained by reversing the direction of time in the equation, considering $u_T$ as terminal condition.
In this case, we use the notion of backward viscosity solution, that provides existence and uniqueness for the terminal-value problem.
We also give an equivalent reachability condition based on a differential inequality, that 
relates the reachability of the target with its semiconcavity properties.
Then, for the case when $u_T$ is reachable, we construct the set of
all initial conditions for which the solution coincides with $u_T$ at time $T$.
Note that in general, such initial conditions are not unique. 
Finally, for the case when the target $u_T$ is not necessarily reachable,
we study the projection of $u_T$ on the set of reachable targets,
obtained by solving the problem backward and then forward in time.
This projection is then identified with the solution of a fully nonlinear obstacle problem,
and can be interpreted as the semiconcave envelope of $u_T$,
i.e. the smallest reachable target bounded from below by $u_T$.
\end{abstract}

\author{Carlos Esteve}
\thanks{AMS 2020 MSC: 35F21, 35F25, 35J70, 49L25 \\
\textit{Keywords: Hamilton-Jacobi equation, inverse design problem, semiconcave envelopes, obstacle problems} \\
\textbf{Funding:} This project has received funding from the European Research Council
(ERC) under the European Union’s Horizon 2020 research and
innovation programme (grant agreement No 694126-DYCON). \\
The second author has received funding from
Transregio 154 Project *Mathematical Modelling, Simulation and Optimization using the Example of Gas Networks* of the German  DFG}
\address {Carlos Esteve \newline \indent
{Departamento de Matem\'aticas, \newline \indent
Universidad Aut\'onoma de Madrid},
\newline \indent
{28049 Madrid, Spain}
\newline \indent \hspace{2.5cm} \text{and} \newline \indent
{Chair of Computational Mathematics, Fundaci\'on Deusto}
\newline \indent
{Av. de las Universidades, 24}
\newline \indent
{48007 Bilbao, Basque Country, Spain}
}
\email{\texttt{carlos.esteve@uam.es}}

\author{Enrique Zuazua}
\address{Enrique Zuazua \newline \indent
{Chair in Applied Analysis, Alexander von Humboldt-Professorship
\newline \indent
Department of Mathematics, \newline \indent
Friedrich-Alexander-Universit\"at Erlangen-N\"urnberg}
\newline \indent
{91058 Erlangen, Germany}
\newline \indent \hspace{2.5cm} \text{and} \newline \indent
{Chair of Computational Mathematics, Fundación Deusto}
\newline \indent
{Av. de las Universidades, 24}
\newline \indent
{48007 Bilbao, Basque Country, Spain}
\newline \indent \hspace{2.5cm} \text{and} \newline \indent
{Departamento de Matemáticas, \newline \indent
Universidad Autónoma de Madrid,}
\newline \indent
{28049 Madrid, Spain}
}
\email{\texttt{enrique.zuazua@fau.de}}

\date{\today}

\maketitle

\section{Introduction}

We consider the initial-value problem for a Hamilton-Jacobi equation of the form
\begin{equation}\label{HamJac eq}
\left\{ \begin{array}{ll}
\partial_t u + H(D_x u) = 0, & \text{in} \ [0,T]\times \R^n, \\
\noalign{\vskip 1.5mm}
u(0,x) = u_0(x), & \text{in} \ \R^n,
\end{array}\right.
\end{equation}
where $u_0\in\Lip(\R^n)$ and the Hamiltonian $H: \R^n\rightarrow\R$ is assumed to satisfy the following hypotheses:
\begin{equation}\label{CondH}
H\in C^2(\R^n), \quad H_{pp}(p) >0, \ \forall p\in\R^n, \quad \text{and} \quad 
\lim_{|p|\to\infty} \dfrac{H(p)}{|p|} = +\infty.
\end{equation}
Here, the unknown $u$ is a function $[0,T]\times\R^n\longrightarrow\R$,
the notation $\partial_t u$ stands for the derivative of $u$ with respect to the first variable
and $D_x u$, for the vector of partial derivatives with respect to the second group of variables. 
The inequality $H_{pp}(p)>0$ in \eqref{CondH} means that the Hessian matrix of $H$ at $p$ is positive definite.
The study of equations such as \eqref{HamJac eq}
arises in the context of optimal control theory and calculus of variations,
where the value function satisfies, in a weak sense, a Hamilton-Jacobi equation like \eqref{HamJac eq}
(see \cite{bardi2008optimal,barles1994solutions,barles2013introduction,cannarsa2004semiconcave,evans2010partial,fleming2006controlled} 
and the references therein).
In this context, Hamilton-Jacobi equations have applications in a wide range of fields such as
economics, physics, mathematical finance, traffic flow and geometrical optics. 

It is well known that, due to the presence of a nonlinear term in the equation, 
one cannot in general expect the existence of a classical $C^1$ solution for problem \eqref{HamJac eq},
even if the initial datum $u_0$ is assumed to be very smooth.
On the other hand, continuous solutions satisfying the equation almost everywhere might not be unique.
In the early 80's, Crandall and Lions solved this problem in \cite{crandall1983viscosity} 
by introducing the notion of \emph{viscosity solution} (Definition \ref{Def visc sol HJ}), see also \cite{barles1994solutions,cannarsa2004semiconcave,crandall1992user,katzourakis2014introduction}.

For any initial condition $u_0\in\Lip(\R^n)$, existence and uniqueness of a viscosity solution
$u\in \Lip ([0,T]\times\R^n)$ was established in \cite{crandall1983viscosity}.
This solution can be obtained, using the method of the \emph{vanishing viscosity} (see \cite{crandall1992user,crandall1983viscosity,lions1982generalized}), 
as the limit when $\varepsilon\to 0^+$ of the unique solution of the parabolic problem
\begin{equation*}
\left\{\begin{array}{ll}
\partial_t u - \varepsilon\, \Delta_x u + H(D_x u) = 0, & \text{in} \ [0,T]\times \R^n, \\
\noalign{\vspace{1.5mm}}
u(0,x) = u_0(x), & \text{in} \ \R^n,
\end{array}\right.
\end{equation*}
for which, by the standard theory of nonlinear parabolic equations, existence and uniqueness of a classical solution holds as long as $\varepsilon>0$.

In view of the existence and uniqueness of a viscosity solution for problem \eqref{HamJac eq},
we can define, for a fixed $T>0$, the following nonlinear operator, which associates to any initial condition $u_0$, the function
$u(T,\cdot)$, where $u$ is the viscosity solution of \eqref{HamJac eq}:
\begin{equation*}
\begin{array}{cccl}
S^+_T : & \Lip(\R^n) & \longrightarrow & \Lip(\R^n) \\
 & u_0 & \longmapsto & S^+_T u_0 := u(T,\cdot)
\end{array}
\end{equation*}

Our goal in this work is to study the inverse problem associated to \eqref{HamJac eq}.
More precisely, for a given target function $u_T$ and a time horizon $T>0$,
we want to construct all the initial conditions $u_0$ such that the viscosity solution of \eqref{HamJac eq} coincides with $u_T$ at time $T$.
This type of problems, also known in the literature as data assimilation problems,
have relevant importance in any kind of evolution models.
For instance in meteorology \cite{ghil1991data,miyakoda1978initialization},
where the climate prediction must take into account not only the observations at the present time, but also at past times.
For related results in the context of Hamilton-Jacobi equations, we refer to \cite{claudel2011convex,colombo2019initial,misztela2020initial} and the references therein.

The first thing one notices when addressing this problem is that not all the Lipschitz targets are reachable.
Indeed, as it is well known, the viscosity solution to \eqref{HamJac eq} is always a semiconcave function (Definition \ref{Def semiconcave semiconvex}). 
Therefore, an obvious necessary (but not sufficient) condition for the reachability of $u_T$ is that it must be a semiconcave function.
We then split the problem in the following three steps:
\begin{enumerate}
\item First, we study the reachability of the target, i.e. the existence of at least one $u_0$ satisfying $S_T^+u_0 =u_T$.
The natural candidate is the one obtained by reversing the direction of time in the equation \eqref{HamJac eq},
considering the target $u_T$ as terminal condition (see the definition of \emph{backward viscosity solution} in Definition \ref{Def Backward soluiton}).
Indeed, it turns out that the initial datum recovered by this method, allows one to determine whether the target is reachable or not (see Theorem \ref{Thm 1st reach cond}).
We also obtain a reachability criterion based on a differential inequality (see Theorem \ref{Thm 2nd reach cond}),
that links the reachability of the target with its semiconcavity properties.
\item Secondly, if the target is reachable, we construct all the initial conditions $u_0$
such that $S_T^+ u_0 =u_T$ (see Theorem \ref{Thm IniData Identif}).
As we will see, such $u_0$ is not in general unique.
The impossibility of uniquely determining the initial condition from the solution at time $T$ 
is a main feature in first-order nonlinear evolution equations like \eqref{HamJac eq}, 
and can be interpreted as a loss of the initial information due to the nonlinear effects and the loss of regularity.
\item Finally, if the target $u_T$ is not reachable, we project it on the set of reachable targets
by solving the problem \eqref{HamJac eq} backward in time and then forward.
As we will see, this projection has some interesting properties and seems to be the most natural one.
For instance, it is the smallest reachable target which is bounded from below by $u_T$
and can be characterized as the viscosity solution of a fully nonlinear obstacle problem (see Theorem \ref{Thm semiconcave envelope}).
\end{enumerate}

The idea of iterating forward and backward resolutions of the model in order to reconstruct
previous states has already been exploited in different types of evolution equations,
and seems to be an effective method to deal with data assimilation problems like the one treated in this paper.
See for example the works \cite{auroux2005back,auroux2008nudging} where a Back and Forth Nudging algorithm is proposed.

The paper is structured as follows:
in Section \ref{Sec Results},
we present and discuss our main results concerning each one of the three points described above.
Section \ref{Sec Examples} is devoted to some examples that illustrate these results.
In Section \ref{Sec Backward solutions}, we introduce the notion of backward viscosity solution,
present some of its well-known properties, 
and then we prove the reachability criterion of Theorem \ref{Thm 1st reach cond}.
In Section \ref{Sec IniDataConstruction}, we give the proof of the results about the construction of initial conditions
for a reachable target, Theorems \ref{Thm IniData Identif} and \ref{Thm X_I charac}.
Section \ref{Sec semiconcave} is devoted to the study of the composition operator $S_T^+\circ S_T^-$,
that can be viewed as a projection of $u_T$ on the set of reachable targets.
In this section we prove Theorem \ref{Thm semiconcave envelope}, that identifies the image of $S_T^+\circ S_T^-$
with the viscosity solution of a fully nonlinear obstacle problem.
Then, we discuss the connection of this obstacle problem with the concave envelope of a function
and give a geometrical interpretation of $S_T^+(S_T^- u_T)$ as the semiconcave envelope of $u_T$.
We end the Section \ref{Sec semiconcave} with the proof of Theorem \ref{Thm 2nd reach cond},
that gives a criterion for the reachability of a target in terms of a differential inequality.
Finally, in Section \ref{Sec Conclusions} we discuss the contributions of this work and present some possible extensions and future perspectives. 

\section{Main results}\label{Sec Results}

\subsection{Reachability criteria}
For a given target $u_T\in \Lip (\R^n)$,
our first goal is to determine whether or not there exists at least one initial condition $u_0$ such that $S_T^+ u_0=u_T$. That is, we want to give a necessary and sufficient condition for the set
\begin{equation}\label{Initial conditions}
I_T (u_T) := \left\{ u_0\in \Lip(\R^n) \, ; \ S_T^+ u_0 = u_T\right\}
\end{equation}
to be nonempty.
In Theorem \ref{Thm 1st reach cond} below, for any $u_T\in \Lip(\R^n)$, we give a reachability condition 
based on the so-called \emph{backward viscosity solution} (see Definition \ref{Def Backward soluiton}).
This notion of solution was already used in \cite{barron1999regularity} in order to study the relation between the regularity of solutions and the time-reversibility of problem \eqref{HamJac eq}.

It is well known (see Section \ref{Sec Backward solutions} for more details) that in the class of backward viscosity solutions,
existence and uniqueness  holds for the terminal-value problem
\begin{equation}\label{Backward HamJac eq}
\left\{ \begin{array}{ll}
\partial_t w + H(D_x w) = 0, & \text{in} \ [0,T]\times \R^n, \\
\noalign{\vskip 1.5mm}
w(T,x) = u_T(x), & \text{in} \ \R^n.
\end{array}\right.
\end{equation}
Actually, analogously to the (forward) viscosity solutions, the backward viscosity solution can be obtained as the limit when $\varepsilon \to 0^+$ of the solution to the problem
\begin{equation*}
\left\{\begin{array}{ll}
\partial_t w + \varepsilon\, \Delta_x w + H(D_x w) = 0, & \text{in} \ [0,T]\times \R^n, \\
\noalign{\vspace{1.5mm}}
w(T,x) = u_T(x), & \text{in} \ \R^n.
\end{array}\right.
\end{equation*}
Hence, we can define the nonlinear operator
\begin{equation*}
\begin{array}{cccl}
S^-_T : & \Lip(\R^n) & \longrightarrow & \Lip(\R^n) \\
 & u_T & \longmapsto & S^-_T u_T := w(0,\cdot)
\end{array}
\end{equation*}
which associates to any terminal condition $u_T$, the function $w(0,\cdot)$,
i.e. the unique backward viscosity solution of \eqref{Backward HamJac eq} at time $0$.

Here we state the first reachability criterion, which identifies the reachable targets with the fix-points of the composition operator $S_T^+\circ S_T^-$.

\goodbreak

\begin{theorem}\label{Thm 1st reach cond}
Let $H$ satisfy \eqref{CondH}, $u_T\in \Lip(\R^n)$ and $T>0$. Then, the set $I_T(u_T)$ defined in \eqref{Initial conditions} is nonempty 
if and only if  $S_T^+ \left( S_T^- u_T\right) = u_T.$ 
\end{theorem}

\goodbreak

Reversing the time in the equation in order to find initial conditions conducting to a given target 
is a natural approach in all kinds of evolution equations.
However, in many cases, the obtained initial condition does not lead the system back to the target.  
In the case of the problem \eqref{HamJac eq}, Theorem \ref{Thm 1st reach cond} ensures that the target is reachable if and only if this technique of reversing the time gives the desired initial condition.

As a drawback of this result, we need to solve first the problem \eqref{Backward HamJac eq} 
and then the problem \eqref{HamJac eq} in order to determine if a target is reachable or not.
Next, for the one-dimensional case, and for the case of a quadratic Hamiltonian in any space-dimension, i.e. 
\begin{equation}\label{quadratic Hamiltonian}
H(p) = \dfrac{\langle A\, p, p\rangle}{2} \quad \text{where $A$ is a definite positive $n\times n$ matrix},
\end{equation} 
we give a reachability criterion based on a differential inequality.

\goodbreak

\begin{theorem}\label{Thm 2nd reach cond}
Let $u_T\in \Lip(\R^n)$ and $T>0$. 
\begin{enumerate}
\item If $H$ satisfies \eqref{CondH} and the space-dimension is $1$,  
the set $I_T(u_T)$ defined in \eqref{Initial conditions} is nonempty 
if and only if  $u_T$ satisfies the inequality 
$$
\partial_{xx} u_T - \left( T\, H_{pp}(\partial_x u_T)\right)^{-1} \leq 0, \qquad \text{in} \ \R
$$
in the viscosity sense (see Definition \ref{Def visc sol diff ineq}).
\item If $H$ is given by \eqref{quadratic Hamiltonian}, 
the set $I_T(u_T)$ defined in \eqref{Initial conditions} is nonempty 
if and only if  $u_T$ satisfies the inequality 
$$
\lambda_n \left[ D^2 u_T - \dfrac{A^{-1}}{T} \right] \leq 0, \qquad \text{in} \ \R^n
$$
in the viscosity sense (see Definition \ref{Def visc sol diff ineq}).
\end{enumerate}
\end{theorem}

\goodbreak

Here, $D^2 u_T$ denotes the Hessian matrix of $u_T$, and the expression  $\lambda_n[X]$ denotes the largest eigenvalue of the $n\times n$ symmetric matrix $X$.
Analogously, we will use $\lambda_1[X]$ to denote the smallest eigenvalue of $X$.

\begin{remark}\label{Rmk 1D equiv}
In the multidimensional case,
Theorem \ref{Thm 2nd reach cond} only applies to quadratic Hamiltonians.
This is due to the fact that
for this case, the Hessian matrix of $H$ is constant over $\R^n$.
In the one-dimensional case, the result can be generalized to any strictly convex $H$,
however, the arguments that we use in the proof do not apply to higher dimensions,
and a similar necessary and sufficient reachability condition do not seem to be straightforward for the case of 
a general convex Hamiltonian in any space-dimension (see Remark \ref{Rmk no general H in multiD}). 

Note in addition that, in the one-dimensional case, the transformation
\begin{equation}\label{transform 1D}
v(t,x) \longmapsto f'(v(t,x))
\end{equation}
allows to reduce the study of any scalar conservation law of the form
$$
\partial_t v + \partial_x (f(v)) = 0,
$$
to the case of Burgers equation (see for example \cite{liard2019inverse})
$$
\partial_t w + \partial_x \left( \dfrac{w^2}{2} \right) = 0.
$$
Then, using the relation between Hamilton-Jacobi equations and scalar conservations laws (see for example \cite{colombo2019initial}),
in one-space dimension we can reduce the study of \eqref{HamJac eq} to the case $H(p) = \dfrac{p^2}{2}$.
\end{remark}

Observe that the reachability criterion of Theorem \ref{Thm 2nd reach cond} does not involve the operators $S_T^+$ and $S_T^-$.
Furthermore, this result relates the reachability of a target $u_T$ with its semiconcavity properties.
We recall that a continuous function $f:\R^n\to \R$ is concave if and only if 
it is a viscosity solution of $\lambda_n [D^2 f]\leq 0$ (see for example the work of Oberman \cite{oberman2007convex}).
In view of this, from Theorem \ref{Thm 2nd reach cond} and 
the properties of semiconcave functions in Proposition \ref{Prop semiconcave characterization} below, 
we can deduce the following result, that relates the reachability of a function $u_T$ with its semiconcavity constant.

\begin{corollary}\label{Cor reach cond}
Let $H$ be given by \eqref{quadratic Hamiltonian}, $u_T\in \Lip(\R^n)$ and $T>0$.
\begin{enumerate}
\item If $I_T(u_T)\neq \emptyset$, then $u_T$ is semiconcave with linear modulus and constant 
$\dfrac{1}{T\lambda_1(A)}$.
\item If $u_T$ is semiconcave with linear modulus and constant $\dfrac{1}{T\lambda_n (A)}$,
then $I_T(u_T)\neq \emptyset$.
\end{enumerate}
\end{corollary}

See a detailed proof of this corollary in Section \ref{Sec semiconcave}.
For the precise definition of semiconcave function with linear modulus, see Definition \ref{Def semiconcave semiconvex}.

\begin{remark}\label{Rmk reach cond}
\begin{enumerate}
\item Observe that, in the particular case of a Hamiltonian given by \eqref{quadratic Hamiltonian} with $A=c\, I_n$ and $c>0$, i.e.
$$H(p) = c\dfrac{|p|^2}{2},$$
we have  $\lambda_1(A)=\lambda_n(A)=c$. 
Then, Corollary \ref{Cor reach cond} implies that
$I_T(u_T)\neq \emptyset$ if and only if $u_T$ is semiconcave with linear modulus and constant $\dfrac{1}{c\, T}$.

\item It can be easily checked that, if a function $u_T$ satisfies the inequality of Theorem \ref{Thm 2nd reach cond} for some $T>0$, 
then the same inequality holds for any $T'\in ]0, T]$.
This implies that the set of reachable targets becomes smaller as we increase the time horizon $T$.
In the limit case, if we let $T$ go to $\infty$, 
we observe that only the concave functions are reachable for all $T>0$. 
\end{enumerate}
\end{remark}

\subsection{Initial data construction}
Here, for the case when the target $u_T$ is reachable,
our goal is to construct all the initial conditions $u_0$ in $I_T(u_T)$.
Our construction relies on the fact that, in view of Theorem \ref{Thm 1st reach cond},
$I_T(u_T)\neq \emptyset$ implies that $S_T^- u_T\in I_T(u_T)$.

\goodbreak

\begin{theorem}\label{Thm IniData Identif}
Let $H$ satisfy \eqref{CondH} and $T>0$.
Let $u_T\in \Lip(\R^n)$ be such that $I_T(u_T)\neq \emptyset$ and set the function $\tilde{u}_0 := S_T^- u_T$.
Then, for any $u_0\in \Lip (\R^n)$, the two following statements are equivalent:
\begin{enumerate}
\item\label{Thm IniData Identif u_0 in I_T} $u_0\in I_T(u_T)$;
\item\label{Thm IniData Identif u0 geq uast} $u_0(x)\geq \tilde{u}_0 (x), \ \forall x\in \R^n \quad \text{and} \quad u_0(x) = \tilde{u}_0(x), \ \forall x\in X_T(u_T),$
\end{enumerate}
where $X_T(u_T)$ is the subset of $\R^n$ given by 
$$
X_T(u_T) : = \left\{ z- T\, H_p (\nabla u_T(z)); \ \forall z\in\R^n \ \text{such that} \ 
u_T(\cdot) \ \text{is differentiable at} \ z\right\}.
$$
\end{theorem}

\goodbreak

See the Examples \ref{Example 3} and \ref{Example 3.5} for an illustration of this result.
In view of Theorem \ref{Thm IniData Identif}, when it is nonempty, the set of initial conditions $I_T(u_T)$ defined in \eqref{Initial conditions}
can be given in the following way:
$$
I_T(u_T) = \left\{ \tilde{u}_0 + \varphi \, ; \, \varphi\in\Lip(\R^n)\ 
\text{such that} \ \varphi\geq 0\ \text{and}\ \text{supp} (\varphi) \subset \R^n\setminus X_T(u_T)\right\}.
$$
All the functions in $I_T(u_T)$ coincide with $\tilde{u}_0$ in the set $X_T(u_T)\subset\R^n$,
while in its complement, they are bigger or equal than $\tilde{u}_0$.
We can also write
$$
I_T(u_T) = \tilde{u}_0 + \mathcal{L}(\R^n \setminus X_T(u_T)),
$$
where, for a subset $A$ of $\R^n$, $\mathcal{L}(A)\subset \Lip (\R^n)$ represents the convex cone defined as
$$
\mathcal{L}(A):= \left\{ \varphi \in \Lip (\R^n) \ \text{such that} \ \varphi\geq 0\ \text{and}\ \text{supp} (\varphi) \subset A\right\}
$$

\begin{remark}
\begin{enumerate}
\item We observe that in order to construct all the elements in $I_T(u_T)$,
we need two ingredients: 
the function $\tilde{u}_0$, that can be obtained 
as the backward viscosity solution of \eqref{Backward HamJac eq} using 
the formula \eqref{Backward Hopf formula};
and  the set $X_T(u_T)$, which can be deduced from the points of differentiability of $u_T$.
In Theorem \ref{Thm X_I charac}, for the case of a quadratic Hamiltonian of the form \eqref{quadratic Hamiltonian}, we give a different characterization of the set $X_T(u_T)$ which does not involve the differentiability points of $u_T$, so that the set of initial conditions $I_T(u_T)$ can be constructed without knowing the set of points where $u_T$ is differentiable.

\item It is important to note that, unlike other models, as for example the heat equation
$$
\partial_t u - \Delta u = 0, \qquad \text{in} \ [0,T]\times \R^n,
$$
for which backward uniqueness holds, 
for the problem \eqref{HamJac eq}, a target $u_T$ can be reached by considering different initial conditions.
Indeed, this is the case whenever $I_T(u_T)\neq \emptyset$ and the set $X_T(u_T)$, introduced in Theorem \ref{Thm IniData Identif},
is a proper subset of $\R^n$.
See Example \ref{Example 3} for an illustration of this phenomenon.

\item Similar results on initial data reconstruction were obtained recently by Colombo and Perrollaz \cite{colombo2019initial}
and by Liard and Zuazua \cite{liard2019inverse} for scalar conservation laws and for Hamilton-Jacobi equations in dimension 1.
In fact, exploiting the relation between Hamilton-Jacobi equations and hyperbolic systems of conservation laws,
our results might be adapted to generalize, to the $n$-dimensional case, the results given in \cite{colombo2019initial,liard2019inverse}. 
\end{enumerate}
\end{remark}

The following theorem gives a different characterization of the set $X_T(u_T)$
introduced in Theorem \ref{Thm IniData Identif}, for the case when $H$ is given by  \eqref{quadratic Hamiltonian}.
This result identifies $X_T(u_T)$ with the set of points for which the functions in $I_T(u_T)$ admit a touching paraboloid from below.

\goodbreak

\begin{theorem}\label{Thm X_I charac}
Let $H$ be given by \eqref{quadratic Hamiltonian} and $T>0$.
Let $u_T\in \Lip(\R^n)$ be such that $I_T(u_T)\neq \emptyset$ and take any function $u_0\in I_T(u_T)$.
Then, for any $x_0\in\R^n$, the following two statements are equivalent:
\begin{enumerate}
\item\label{Thm X_I charac i} $x_0\in X_T(u_T)$;
\item\label{Thm X_I charac ii} There exist $b\in \R^n$ and $c\in \R$ such that 
$$
u_0 (x_0) = -\dfrac{ \langle A^{-1}x_0, x_0\rangle}{2T} +  b\cdot x_0 + c \qquad \text{and}
$$
$$
u_0 (x) > -\dfrac{ \langle A^{-1}x,  x \rangle}{2T} + b \cdot x + c, \qquad \forall x\in \R^n\setminus \{x_0\}.
$$
\end{enumerate}
\end{theorem}

\goodbreak

Note that condition \eqref{Thm X_I charac ii} in this theorem is independent of the choice of $u_0\in I_T(u_T)$.
Although the initial condition $\tilde{u}_0 = S_T^- u_T$,
obtained by pulling back the target with the operator $S_T^-$, seems to be a natural choice to identify the points in $X_T(u_T)$, where in view of Theorem \ref{Thm IniData Identif}, all the initial conditions in $I_T(u_T)$ coincide,
we point out that any element in $I_T(u_T)$ would suffice to carry out this
construction.
Another important advantage of this theorem is that, unlike in Theorem \ref{Thm IniData Identif},
the set $X_T(u_T)$ is characterized independently of the points of differentiability of $u_T$.

\subsection{Projection on the set of reachable targets and semiconcave envelopes}
In this subsection, we treat the case when the target $u_T$ is not necessarily reachable.
Recall that if, for instance, $u_T$ is not a semiconcave function, then $I_T(u_T)=\emptyset$
(see Proposition \ref{Prop semiconcave}).
Let us introduce the following composition operator
\begin{equation*}
\begin{array}{cccc}
S_T^+\circ S^-_T : & \Lip(\R^n) & \longrightarrow & \Lip(\R^n) \\
 & u_T & \longmapsto & S_T^+(S^-_T u_T).
\end{array}
\end{equation*}
This operator can be viewed as a projection of $\Lip (\R^n)$ onto the set of reachable targets.

Throughout the paper, for any $u_T\in\Lip(\R^n)$, we will denote
\begin{equation}\label{projection on reachability}
u_T^\ast : = S_T^+(S_T^- u_T),
\end{equation}
the projection of $u_T$ on the set of reachable targets.
Observe that for any $u_T\in\Lip(\R^n)$, the set $I_T(u_T^\ast)$ is nonempty. 
Indeed, by definition, the initial condition $u_0 := S_T^- u_T$ belongs to $I_T(u_T^\ast)$.

As well as Theorem \ref{Thm 2nd reach cond}, our main result in this subsection applies to the one-dimensional case
and to the case of a quadratic Hamiltonian $H$ in any space dimension.
We prove that for any $u_T\in\Lip(\R^n)$, the function $u_T^\ast$ defined in \eqref{projection on reachability} is the viscosity solution of the following fully nonlinear obstacle problem:
\begin{equation}\label{obstacle problem general}
\min \left\{ v - u_T, \ -\lambda_n \left[ D^2 v - \dfrac{[H_{pp}(D v)]^{-1}}{T}\right] \right\} = 0.
\end{equation}
See Definition \ref{Def vis sol obstacle} for the precise definition of viscosity solution to this equation.
We recall that for a $n\times n$ symmetric matrix $X$, 
$\lambda_1[X]$ and $\lambda_n[X]$ 
denote respectively the smallest and the largest eigenvalues of $X$.

Note that in the one-dimensional case,
the equation \eqref{obstacle problem general} can be simply written as
\begin{equation}\label{obstacle problem 1D}
\min \left\{ v - u_T, \ - \partial_{xx} v + \left( T\, H_{pp} (\partial_x v) \right)^{-1} \right\} = 0,
\end{equation}
while in the case of a quadratic Hamiltonian given by \eqref{quadratic Hamiltonian} in any space dimension, 
equation \eqref{obstacle problem general} can be written as
\begin{equation}\label{obstacle problem}
\min \left\{ v - u_T, \ -\lambda_n \left[ D^2 v - \dfrac{A^{-1}}{T} \right] \right\} = 0.
\end{equation}

Analogously to the obstacle problem for the convex envelope, introduced by Oberman in \cite{oberman2007convex}, 
for a given function $f:\R^n \rightarrow \R$, 
the \emph{concave envelope} of $f$ in $\R^n$ is the unique
viscosity solution of the obstacle problem
\begin{equation}\label{concave envelope eq}
\min \{ v - f, \ -\lambda_n[D^2v]\} =0.
\end{equation}
We recall that the concave envelope of $f$ is the smallest concave function which is bounded from below by $f$, i.e.
\begin{equation}\label{convex envelope def}
f^\ast (x):= \inf \{ v(x)\, ; \, v \ \text{concave}, \ v(y)\geq f, \ \text{for all} \ y\in \R^n\}.
\end{equation}
See Figure \ref{FigConcaveEnvelope} for an illustration of a function and its concave envelope.

\begin{remark}\label{Rmk concave Perron}
\begin{enumerate}
\item As we shall prove in Lemma \ref{Lemma semiconcave concave}, for any $T>0$, $u_T\in\Lip (\R^n)$ and any positive definite matrix $A$,
a function $v$ is a viscosity solution of \eqref{obstacle problem}
if and only if the function
$$
w(x):= v(x) - \dfrac{\langle A^{-1}x,x\rangle}{2T} 
$$ 
is a viscosity solution of \eqref{concave envelope eq} with
$$
f(x):= u_T (x) - \dfrac{\langle A^{-1}x, x\rangle}{2T}.
$$
In other words, $v$ is a viscosity solution of \eqref{obstacle problem} if and only if
$w$ is the concave envelope of the function $f$ above defined.
It gives an alternative way to obtain the viscosity solution of \eqref{obstacle problem}
in terms of the concave envelope of the function $f$.

\item In the one-dimensional case, the study of problem \eqref{HamJac eq} for any strictly convex $H$
can be reduced, after a transformation (see Remark \ref{Rmk 1D equiv}), to the case $H(p)= p^2/2$.
Then, equation \eqref{obstacle problem 1D} can also be reduced to equation \eqref{obstacle problem} in dimension one with $A=1$.

\item Since the concave envelope of a function is unique,
we deduce uniqueness of a viscosity solution for problems \eqref{obstacle problem 1D} and \eqref{obstacle problem}.
The existence of a viscosity solution 
can be obtained from Theorem \ref{Thm semiconcave envelope} by applying the operator $S_T^+\circ S_T^-$
to the function $u_T$.
We note that existence and uniqueness can also be deduced directly by means of the Perron's method (see \cite{crandall1992user}).
\end{enumerate}
\end{remark}

In analogy with the notion of concave envelope, for any $T>0$ and any positive definite $n\times n$ matrix $A$, 
we will refer to the viscosity solution of \eqref{obstacle problem}
as the $\frac{A^{-1}}{T}-$semiconcave envelope of $u_T$ in $\R^n$.
Note that being a viscosity solution of \eqref{obstacle problem}
implies in particular semiconcavity with linear modulus and constant 
$\dfrac{\lambda_n(A^{-1})}{T} = \dfrac{1}{T\, \lambda_1(A)}$.
See Subsection \ref{Sec semiconcave 1.2} for a justification of this fact.

Here we state the main result of this subsection, which ensures that the function $u_T^\ast$ defined in \eqref{projection on reachability} is the $\frac{A^{-1}}{T}-$semiconcave envelope of $u_T$ in $\R^n$.

\goodbreak

\begin{theorem}\label{Thm semiconcave envelope}
Let $u_T\in \Lip(\R^n)$ and $T>0$.
Then,  
\begin{enumerate}
\item If $H$ satisfies \eqref{CondH} and the space-dimension is $1$, 
the function $u_T^\ast := S_T^+(S_T^- u_T)$ is the unique viscosity solution of \eqref{obstacle problem 1D}.

\item If $H$ is given by \eqref{quadratic Hamiltonian}, 
the function $u_T^\ast := S_T^+(S_T^- u_T)$ is the unique viscosity solution of \eqref{obstacle problem}.
\end{enumerate}
\end{theorem}

\goodbreak

See an illustration of this result in Examples \ref{Ex4} and \ref{Ex5}.
As a consequence of this theorem and Lemma \ref{Lemma semiconcave concave}, we deduce the following corollary
which, for the case of a quadratic Hamiltonian, gives an alternative way to obtain $u_T^\ast$ in terms of the concave envelope of a certain function.
This allows one to compute the projection of $u_T$ on the set of reachable targets without applying the operators
$S_T^-$ and $S_T^+$.

\begin{corollary}\label{Cor concave envelope}
Let $H$ be given by \eqref{quadratic Hamiltonian}, $u_T\in \Lip(\R^n)$ and $T>0$.
Let $f^\ast$ be the concave envelope of
$$
f(x):= u_T (x) - \dfrac{\langle A^{-1}x, x\rangle}{2T}. 
$$
Then, the function $u_T^\ast$ defined in \eqref{projection on reachability} satisfies
$$
u_T^\ast (x) = f^\ast (x) + \dfrac{\langle A^{-1}x,x\rangle}{2T}, \qquad \text{for all} \ x\in\R^n.
$$
\end{corollary}
 
For the particular case 
$$H(p)=c\frac{|p|^2}{2}, \qquad \text{with} \ c>0,$$ 
the function $u_T^\ast := S_T^+(S_T^- u_T)$ is the viscosity solution of
\begin{equation*}
\min \left\{ v - u_T, \ -\lambda_n \left[ D^2 v\right] + \dfrac{1}{c\, T} \right\} = 0.
\end{equation*}
This can be deduced from Theorem \ref{Thm semiconcave envelope} and Property \ref{Prop eigenvalues}, together with the fact that
$\lambda_1(c^{-1}I_n)=\lambda_n(c^{-1}I_n)=c^{-1}$.
In this case, $u_T^\ast$ can be identified with the $\frac{1}{c\, T}-$semiconcave envelope of $u_T$ in $\R^n$,
that is, 
the smallest semiconcave function with linear modulus and constant $\frac{1}{c\, T}$  which is bounded from below by $u_T$.

\section{Examples}\label{Sec Examples}

\begin{example}\label{Example 3}
Here we give a particular example of application of Theorem \ref{Thm IniData Identif}.
We consider the one-dimensional case and the Hamiltonian $H(p) = |p|^2/2$.
As time horizon and reachable target we choose $T=0.5$ and
\begin{equation}\label{Ex 3}
u_T(x):= S^+_T u_1(x), \quad \text{where} \quad
u_1(x):=\left\{
\begin{array}{ll}
1-|x+1| & \text{if} \ -2 <x\leq 0 \\
1-|x-1| & \text{if} \ 0 < x<2 \\
0 & \text{else.}
\end{array}
\right.
\end{equation}
Note that $I_T(u_T)\neq \emptyset$, indeed, $u_1\in I_T(u_T)$.
In Figure \ref{Fig3a}, we can see the function $u_T$.
We observe that $u_T$ is differentiable at all points except for $-1$ and $1$.
Computing the lateral derivatives at these two points, the set $X_T(u_T)$ can be easily determined:
$$
X_T(u_T) = \R\setminus \left([-1.5,-0.5]\cup[0.5,1.5] \right).
$$

The function $\tilde{u}_0 := S_T^- u_T$ is represented Figure \ref{Fig3b}.
The restriction of $\tilde{u}_0$ to the set $X_T(u_T)$ is marked by a red line. 
In view of Theorem \ref{Thm IniData Identif}, the functions in $I_T(u_T)$ are those which coincide
with this function on the red line, while they are bigger or equal than it on the black line.
In the same plot, we can also see the function $u_1$, represented by a dotted line,
as another element in $I_T(u_T)$ different to $\tilde{u}_0$.

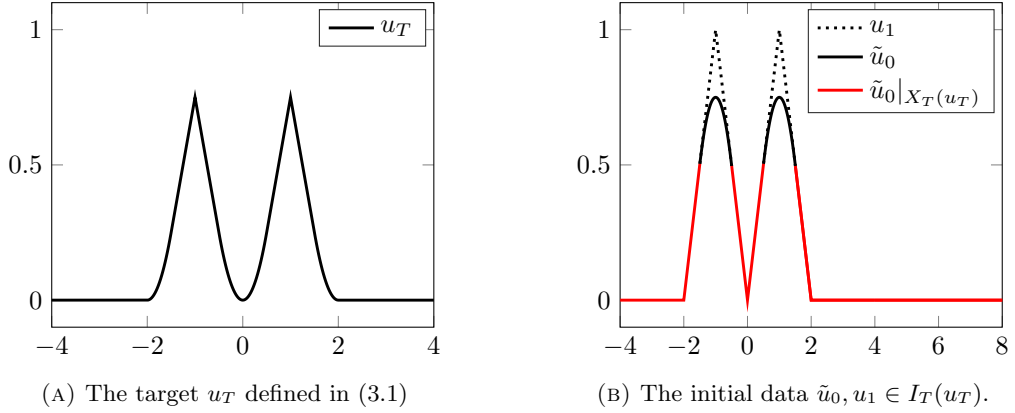
\begin{figure}[h]
\centering
\begin{subfigure}{.5\textwidth}
  \centering
  \begin{tikzpicture}

\begin{axis}[%
width=2in,
height=1.7in,
at={(0.758in,0.481in)},
scale only axis,
xmin=-4,
xmax=4,
ymin=-0.1,
ymax=1.1,
axis background/.style={fill=white},
legend style={legend cell align=left, align=left, draw=white!15!black}
]
\addplot [color=black, line width=1.1pt]
  table[row sep=crcr]{%
-4	0\\
-1.984	0.000256000000000256\\
-1.96	0.00159999999999982\\
-1.936	0.00409599999999966\\
-1.912	0.00774399999999975\\
-1.888	0.0125440000000001\\
-1.864	0.0184959999999998\\
-1.84	0.0255999999999998\\
-1.816	0.0338560000000001\\
-1.792	0.0432639999999997\\
-1.768	0.0538240000000005\\
-1.744	0.0655359999999998\\
-1.72	0.0784000000000002\\
-1.696	0.0924160000000001\\
-1.672	0.107584\\
-1.648	0.123904\\
-1.624	0.141376\\
-1.6	0.16\\
-1.576	0.179776\\
-1.552	0.200704\\
-1.528	0.222784\\
-1.504	0.246016\\
-1	0.750016\\
-0.48	0.2304\\
-0.456	0.207936\\
-0.432	0.186624\\
-0.408	0.166464\\
-0.384	0.147456\\
-0.36	0.1296\\
-0.336	0.112896\\
-0.312	0.0973439999999997\\
-0.288	0.0829439999999995\\
-0.264	0.0696959999999995\\
-0.24	0.0575999999999999\\
-0.216	0.0466559999999996\\
-0.192	0.0368639999999996\\
-0.168	0.0282239999999998\\
-0.144	0.0207360000000003\\
-0.12	0.0144000000000002\\
-0.0960000000000001	0.00921600000000034\\
-0.0720000000000001	0.00518399999999986\\
-0.048	0.00230399999999964\\
-0.024	0.000575999999999688\\
0	0\\
0.024	0.000575999999999688\\
0.048	0.00230399999999964\\
0.0720000000000001	0.00518399999999986\\
0.0960000000000001	0.00921600000000034\\
0.12	0.0144000000000002\\
0.144	0.0207360000000003\\
0.168	0.0282239999999998\\
0.192	0.0368639999999996\\
0.216	0.0466559999999996\\
0.24	0.0575999999999999\\
0.264	0.0696959999999995\\
0.288	0.0829439999999995\\
0.312	0.0973439999999997\\
0.336	0.112896\\
0.36	0.1296\\
0.384	0.147456\\
0.408	0.166464\\
0.432	0.186624\\
0.456	0.207936\\
0.48	0.2304\\
0.504	0.254016\\
1	0.750016\\
1.52	0.2304\\
1.544	0.207936\\
1.568	0.186624\\
1.592	0.166464\\
1.616	0.147456\\
1.64	0.1296\\
1.664	0.112896\\
1.688	0.0973439999999997\\
1.712	0.0829440000000004\\
1.736	0.0696960000000004\\
1.76	0.0575999999999999\\
1.784	0.0466560000000005\\
1.808	0.0368640000000005\\
1.832	0.0282239999999998\\
1.856	0.0207360000000003\\
1.88	0.0144000000000002\\
1.904	0.00921600000000034\\
1.928	0.00518399999999986\\
1.952	0.00230399999999964\\
1.976	0.000575999999999688\\
2	0\\
4	0\\
};
\addlegendentry{$u_T$}

\end{axis}
\end{tikzpicture}%
  \caption{The target $u_T$ defined in \eqref{Ex 3}}
  \label{Fig3a}
\end{subfigure}%
\begin{subfigure}{.5\textwidth}
  \centering
  \begin{tikzpicture}

\begin{axis}[%
width=2in,
height=1.7in,
at={(0.758in,0.481in)},
scale only axis,
xmin=-4,
xmax=8,
ymin=-0.1,
ymax=1.1,
axis background/.style={fill=white},
legend style={legend cell align=left, align=left, draw=white!15!black}
]
\addplot [color=black, dotted, line width=1.1pt]
  table[row sep=crcr]{%
-1.496	0.504\\
-1	1\\
-0.496	0.496\\
};
\addlegendentry{$u_1$}

\addplot [color=black, line width=1.1pt]
  table[row sep=crcr]{%
-1.496	0.504\\
-1.472	0.527232\\
-1.448	0.549312\\
-1.424	0.57024\\
-1.4	0.590016\\
-1.376	0.60864\\
-1.352	0.626112\\
-1.328	0.642432\\
-1.304	0.6576\\
-1.28	0.671616\\
-1.256	0.68448\\
-1.232	0.696192\\
-1.208	0.706752\\
-1.184	0.71616\\
-1.16	0.724416\\
-1.136	0.73152\\
-1.112	0.737472\\
-1.088	0.742272\\
-1.064	0.74592\\
-1.04	0.748416\\
-1.016	0.74976\\
-0.992	0.749952\\
-0.968	0.748992\\
-0.944	0.74688\\
-0.92	0.743616\\
-0.896	0.7392\\
-0.872	0.733632\\
-0.848	0.726912\\
-0.824	0.71904\\
-0.8	0.710016\\
-0.776	0.69984\\
-0.752	0.688512\\
-0.728	0.676032\\
-0.704	0.6624\\
-0.68	0.647616\\
-0.656	0.63168\\
-0.632	0.614592\\
-0.608	0.596352\\
-0.584	0.57696\\
-0.56	0.556416\\
-0.536	0.53472\\
-0.512	0.511872\\
-0.496	0.496\\
};
\addlegendentry{$\tilde{u}_0$}

\addplot [color=red, line width=1.1pt]
  table[row sep=crcr]{%
-4	0\\
-2	0\\
-1.496	0.504\\
};
\addlegendentry{$\tilde{u}_0|_{X_T(u_T)}$}

\addplot [color=black, dotted, line width=1.1pt, forget plot]
  table[row sep=crcr]{%
0.504	0.504\\
1	1\\
1.504	0.496\\
};
\addplot [color=black, line width=1.1pt, forget plot]
  table[row sep=crcr]{%
0.504	0.504\\
0.528	0.527232\\
0.552	0.549312\\
0.576	0.57024\\
0.6	0.590016\\
0.624	0.60864\\
0.648	0.626112\\
0.672	0.642432\\
0.696	0.6576\\
0.72	0.671616\\
0.744	0.68448\\
0.768	0.696192\\
0.792	0.706752\\
0.816	0.71616\\
0.84	0.724416\\
0.864	0.73152\\
0.888	0.737472\\
0.912	0.742272\\
0.936	0.74592\\
0.96	0.748416\\
0.984	0.74976\\
1.008	0.749952\\
1.032	0.748992\\
1.056	0.74688\\
1.08	0.743616\\
1.104	0.7392\\
1.128	0.733632\\
1.152	0.726912\\
1.176	0.71904\\
1.2	0.710016\\
1.224	0.69984\\
1.248	0.688512\\
1.272	0.676032\\
1.296	0.6624\\
1.32	0.647616\\
1.344	0.63168\\
1.368	0.614592\\
1.392	0.596352\\
1.416	0.57696\\
1.44	0.556416\\
1.464	0.53472\\
1.488	0.511872\\
1.504	0.496\\
};
\addplot [color=red, line width=1.1pt, forget plot]
  table[row sep=crcr]{%
-0.496	0.496\\
0	0\\
0.504	0.504\\
};
\addplot [color=red, line width=1.1pt, forget plot]
  table[row sep=crcr]{%
1.504	0.496\\
2	0\\
4	0\\
};
\addplot [color=red, line width=1.1pt, forget plot]
  table[row sep=crcr]{%
4	0\\
8	0\\
};
\addplot [color=red, line width=1.1pt, forget plot]
  table[row sep=crcr]{%
1.504	0.496\\
2	0\\
4	0\\
};
\addplot [color=red, line width=1.1pt, forget plot]
  table[row sep=crcr]{%
4	0\\
8	0\\
};
\end{axis}

\end{tikzpicture}%
  \caption{The initial data $\tilde{u}_0,u_1\in I_T(u_T)$.}
  \label{Fig3b}
\end{subfigure}
\caption{The initial data $\tilde{u}_0$ and $u_1$ satisfy $S_T^+ \tilde{u}_0 = S_T^+ u_1 =u_T$.}
\end{figure}

\end{example}

\begin{example}\label{Example 3.5}
In this example, we give an illustration of the result of Theorem \ref{Thm IniData Identif} for the two-dimensional case. 
We consider the Hamiltonian
$H(p) = |p|^2/2,$
as time horizon we have chosen $T=0.5$, and as target, the function
\begin{equation}\label{Ex 3.5}
u_T:= S_T^+ u_2, \quad \text{where} \quad  
u_2 (x,y):= \left\{ \begin{array}{ll}
|(x,y) - (-2,0)|-1, & \text{if} \ |(x,y) - (-2,0)| < 1 \\
1 -  |(x,y) - (2,0)|, & \text{if} \ |(x,y) - (2,0)| < 1 \\
0,   &   \text{else.}
\end{array}\right.
\end{equation}

Note that $u_T$ is reachable, indeed, $u_2\in I_T(u_T)$.
We have computed numerically the target $u_T = S_T^+ u_2$, the function
$\tilde{u}_0 = S_T^- u_T$ and the set of points $X_T(u_T)$ defined in Theorem \ref{Thm IniData Identif}.
In Figure \ref{Fig3.5a}, we can see that the function $u_T$ has a bump and a well. 
This function is differentiable at all points except for the top of the bump and the circumference around the well. This is due to the semiconcavity of $u_T$ (see Corollary \ref{Cor reach cond} and Remark \ref{Rmk reach cond}).
We have computed, numerically, the projection on $\R^2$ of all the points where $u_T$ is differentiable, by the map
$$
z\longmapsto z-T\, H_p(\nabla u_T(z)).
$$
We have then obtained the set $X_T(u_T)$, represented by the coloured region in Figure \ref{Fig3.5c}.
Finally, we have computed the function $\tilde{u}_0 = S_T^- u_T$ by using formula \eqref{Backward Hopf formula}.
It is represented in Figure \ref{Fig3.5b}.

Now, we can apply Theorem \ref{Thm IniData Identif} in order to construct 
the set of all the initial conditions in $I_T(u_T)$.
The functions in $I_T(u_T)$ are those which coincide with $\tilde{u}_0$ on the set $X_T(u_T)$ (coloured region in Figure \ref{Fig3.5c}), while on its complement (white region) they are bigger or equal than $\tilde{u}_0$.

\begin{figure}[h]
\centering
\begin{subfigure}{.5\textwidth}
  \centering
  \includegraphics[width=4.5cm, height=2.8cm]{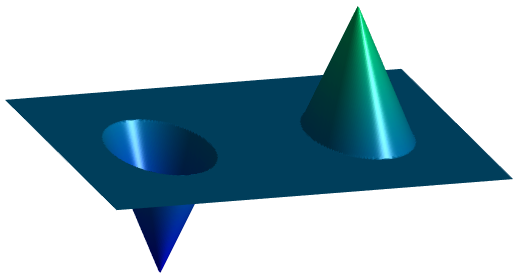}
  \caption{The function $u_2$ defined \\
  in \eqref{Ex 3.5}.}
  \label{Fig3.5d}
\end{subfigure}%
\begin{subfigure}{.5\textwidth}
  \centering
  \includegraphics[width=4.5cm, height=2.8cm]{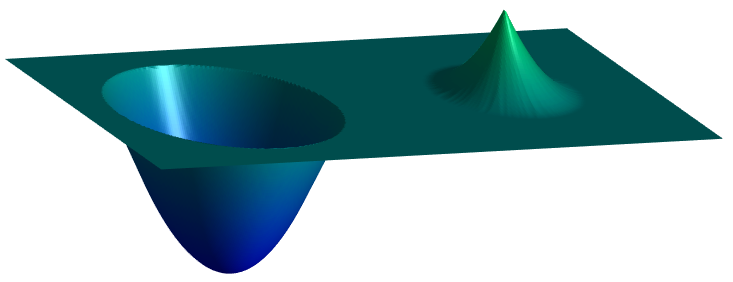}
  \caption{The target $u_T$ defined \\
 as $u_T := S_T^+ u_2$.}
  \label{Fig3.5a}
\end{subfigure}
\begin{subfigure}{.5\textwidth}
  \centering
  \includegraphics[width=4.7cm, height=2.9cm]{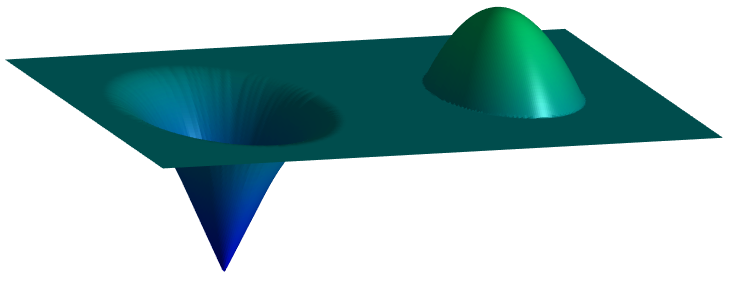}
  \caption{The initial datum $\tilde{u}_0$.}
  \label{Fig3.5b}
\end{subfigure}%
\begin{subfigure}{.5\textwidth}
  \centering
  \includegraphics[width=4.7cm, height=2.9cm]{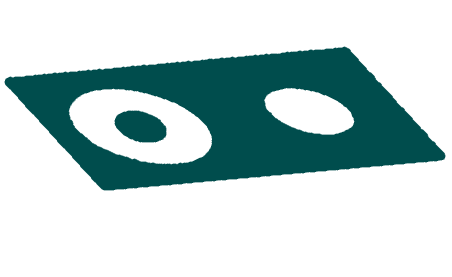}
  \caption{The set $X_T(u_T)$ in dark green.}
  \label{Fig3.5c}
\end{subfigure}
\caption{The initial conditions in $I_T(u_T)$ are those functions which coincide with $\tilde{u}_0$ on the blue region while on its complement they are bigger or equal than it.}
\end{figure}
\end{example}

\begin{example}\label{Ex4}
Here, we illustrate the result of Theorem \ref{Thm semiconcave envelope} for the one-dimensional case and the Hamiltonian $H(p) = |p|^2$.
We have computed the $\frac{1}{T}-$semiconcave envelopes of the functions $u_3$ and $u_4$, defined by
\begin{equation}\label{Example4}
u_3(x) := \left\{ \begin{array}{ll}
|x+1| - 1 & \text{if} \ -2 <x\leq 0 \\
|x-1| - 1 & \text{if} \ 0 < x<2 \\
0 & \text{else.}
\end{array}\right.
\end{equation}

\begin{equation}\label{Example4.5}
u_4(x) := \left\{ \begin{array}{ll}
1-2|x+1| & \text{if} \ -1.5 <x\leq 0 \\
1-2|x-1| & \text{if} \ 0 < x<1.5 \\
0 & \text{else.}
\end{array}\right.
\end{equation} 
with $T=1$ and $T=0.5$ respectively. 

In Figure \ref{Fig4a}, we can see the function $u_3^\ast : = S_T^+(S_T^- u_3)$,
with $T=1$.
In Figure \ref{Fig4b}, we can see the function $u_4^\ast : = S_T^+(S_T^- u_4)$,
with $T=0.5$.
In both plots, the functions $u_3$ and $u_4$ are represented by a dotted line.

We recall that $u_3^\ast$ and $u_4^\ast$ are the projection of $u_3$ and $u_4$ on the set of reachable targets,
and in view of Theorem \ref{Thm semiconcave envelope}, 
they are the viscosity solution of the obstacle problem \eqref{obstacle problem} with $u_T = u_3$ (resp. $u_4$).
The function $u_3^\ast$ (resp. $u_4^\ast$) has been obtained by solving numerically the problem \eqref{Backward HamJac eq} with $u_T = u_3$ (resp. $u_T = u_4$), using formula \eqref{Backward Hopf formula}, and then the problem \eqref{HamJac eq} with $u_0 = S_T^- u_3$ (resp. $u_0 = S_T^- u_4$), using formula \eqref{Hopf formula}.

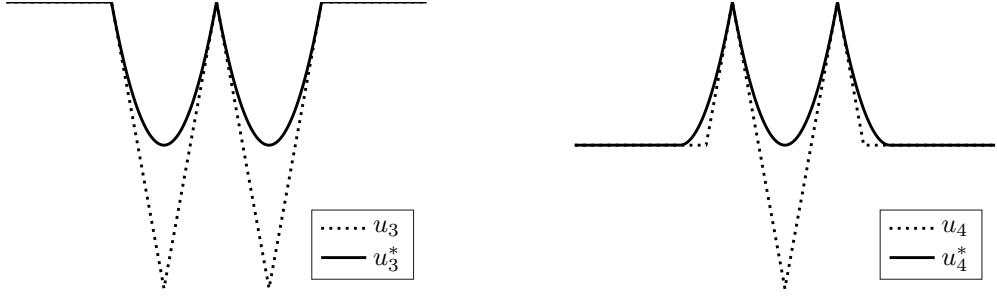
\begin{figure}[h]
\centering
\begin{subfigure}{.5\textwidth}
  \centering
  \begin{tikzpicture}

\begin{axis}[%
width=2.2in,
height=1.5in,
at={(0.758in,0.481in)},
scale only axis,
xmin=-4,
xmax=4,
ymin=-1,
ymax=0,
axis line style={draw=none},
ticks=none,
legend style={at={(0.97,0.03)}, anchor=south east, legend cell align=left, align=left, draw=white!15!black}
]
\addplot [color=black, dotted, line width=1.1pt]
  table[row sep=crcr]{%
-4	0\\
-2	0\\
-1	-1\\
0	0\\
1	-1\\
2	0\\
4	0\\
};
\addlegendentry{$u_3$}

\addplot [color=black, line width=1.1pt]
  table[row sep=crcr]{%
-4	0\\
-2	0\\
-1.968	-0.0314880000000004\\
-1.936	-0.0619519999999998\\
-1.904	-0.0913919999999999\\
-1.872	-0.119808\\
-1.84	-0.1472\\
-1.808	-0.173568\\
-1.776	-0.198912\\
-1.744	-0.223232\\
-1.712	-0.246528000000001\\
-1.68	-0.2688\\
-1.648	-0.290048\\
-1.616	-0.310272\\
-1.584	-0.329472\\
-1.552	-0.347648\\
-1.52	-0.3648\\
-1.488	-0.380928\\
-1.456	-0.396032\\
-1.424	-0.410112\\
-1.392	-0.423168\\
-1.36	-0.4352\\
-1.328	-0.446208\\
-1.296	-0.456192\\
-1.264	-0.465152\\
-1.232	-0.473088\\
-1.2	-0.48\\
-1.168	-0.485888\\
-1.136	-0.490752\\
-1.104	-0.494592\\
-1.08	-0.4968\\
-1.056	-0.498432\\
-1.032	-0.499488\\
-1.008	-0.499968\\
-0.984	-0.499872\\
-0.96	-0.4992\\
-0.936	-0.497952\\
-0.912	-0.496128\\
-0.888	-0.493728\\
-0.864	-0.490752\\
-0.832	-0.485888\\
-0.8	-0.48\\
-0.768	-0.473088\\
-0.736	-0.465152\\
-0.704	-0.456192\\
-0.672	-0.446207999999999\\
-0.640000000000001	-0.4352\\
-0.608000000000001	-0.423168\\
-0.576000000000001	-0.410112\\
-0.544	-0.396032\\
-0.512	-0.380928\\
-0.48	-0.3648\\
-0.448	-0.347648\\
-0.416	-0.329472\\
-0.384	-0.310272\\
-0.352	-0.290048\\
-0.32	-0.2688\\
-0.288	-0.246528\\
-0.256	-0.223231999999999\\
-0.224	-0.198912\\
-0.192	-0.173568\\
-0.16	-0.1472\\
-0.128	-0.119808\\
-0.0960000000000001	-0.0913919999999999\\
-0.0640000000000001	-0.0619519999999998\\
-0.032	-0.0314880000000004\\
0	0\\
0.032	-0.0314880000000004\\
0.0640000000000001	-0.0619519999999998\\
0.0960000000000001	-0.0913919999999999\\
0.128	-0.119808\\
0.16	-0.1472\\
0.192	-0.173568\\
0.224	-0.198912\\
0.256	-0.223231999999999\\
0.288	-0.246528\\
0.32	-0.2688\\
0.352	-0.290048\\
0.384	-0.310272\\
0.416	-0.329472\\
0.448	-0.347648\\
0.48	-0.3648\\
0.512	-0.380928\\
0.544	-0.396032\\
0.576000000000001	-0.410112\\
0.608000000000001	-0.423168\\
0.640000000000001	-0.4352\\
0.672	-0.446207999999999\\
0.704	-0.456192\\
0.736	-0.465152\\
0.768	-0.473088\\
0.8	-0.48\\
0.832	-0.485888\\
0.864	-0.490752\\
0.896	-0.494592\\
0.92	-0.4968\\
0.944	-0.498432\\
0.968	-0.499488\\
0.992	-0.499968\\
1.016	-0.499872\\
1.04	-0.4992\\
1.064	-0.497952\\
1.088	-0.496128\\
1.112	-0.493728\\
1.136	-0.490752\\
1.168	-0.485888\\
1.2	-0.48\\
1.232	-0.473088\\
1.264	-0.465152\\
1.296	-0.456192\\
1.328	-0.446208\\
1.36	-0.4352\\
1.392	-0.423168\\
1.424	-0.410112\\
1.456	-0.396032\\
1.488	-0.380928\\
1.52	-0.3648\\
1.552	-0.347648\\
1.584	-0.329472\\
1.616	-0.310272\\
1.648	-0.290048\\
1.68	-0.2688\\
1.712	-0.246528000000001\\
1.744	-0.223232\\
1.776	-0.198912\\
1.808	-0.173568\\
1.84	-0.1472\\
1.872	-0.119808\\
1.904	-0.0913919999999999\\
1.936	-0.0619519999999998\\
1.968	-0.0314880000000004\\
2	0\\
4	0\\
};
\addlegendentry{$u_3^\ast$}

\end{axis}
\end{tikzpicture}%
  \caption{The $1-$semiconcave envelope of the function \\
   $u_3$ defined in \eqref{Example4}.}
  \label{Fig4a}
\end{subfigure}%
\begin{subfigure}{.5\textwidth}
  \centering
  \begin{tikzpicture}

\begin{axis}[%
width=2.2in,
height=1.5in,
at={(0.758in,0.481in)},
scale only axis,
xmin=-4,
xmax=4,
ymin=-1,
ymax=1,
axis line style={draw=none},
ticks=none,
legend style={at={(0.97,0.03)}, anchor=south east, legend cell align=left, align=left, draw=white!15!black}
]
\addplot [color=black, dotted, line width=1.1pt]
  table[row sep=crcr]{%
-4	-0\\
-1.504	-0\\
-1.496	0.00800000000000001\\
-1	1\\
0	-1\\
1	1\\
1.496	0.00800000000000001\\
1.504	-0\\
4	-0\\
};
\addlegendentry{$u_4$}

\addplot [color=black, line width=1.1pt]
  table[row sep=crcr]{%
-4	0\\
-1.976	0.000575999999999688\\
-1.952	0.00230399999999964\\
-1.928	0.00518399999999986\\
-1.896	0.0108160000000002\\
-1.864	0.0184959999999998\\
-1.832	0.0282239999999998\\
-1.8	0.04\\
-1.768	0.0538240000000005\\
-1.736	0.0696960000000004\\
-1.704	0.0876160000000006\\
-1.672	0.107584\\
-1.64	0.1296\\
-1.608	0.153664\\
-1.576	0.179776\\
-1.544	0.207936\\
-1.512	0.238144\\
-1.48	0.2704\\
-1.448	0.304704\\
-1.416	0.341056\\
-1.384	0.379456\\
-1.352	0.419904\\
-1.32	0.462400000000001\\
-1.28	0.518400000000001\\
-1.24	0.5776\\
-1.2	0.640000000000001\\
-1.16	0.7056\\
-1.12	0.7744\\
-1.08	0.8464\\
-1.04	0.9216\\
-1	1\\
-0.96	0.9216\\
-0.92	0.8464\\
-0.88	0.7744\\
-0.84	0.7056\\
-0.8	0.64\\
-0.76	0.577599999999999\\
-0.72	0.5184\\
-0.68	0.4624\\
-0.640000000000001	0.4096\\
-0.608000000000001	0.369664\\
-0.576000000000001	0.331776000000001\\
-0.544	0.295936\\
-0.512	0.262144\\
-0.48	0.2304\\
-0.448	0.200704\\
-0.416	0.173056\\
-0.384	0.147456\\
-0.352	0.123904\\
-0.32	0.1024\\
-0.288	0.0829439999999995\\
-0.256	0.0655359999999998\\
-0.224	0.0501759999999996\\
-0.192	0.0368639999999996\\
-0.16	0.0255999999999998\\
-0.128	0.0163840000000004\\
-0.0960000000000001	0.00921600000000034\\
-0.0640000000000001	0.00409599999999966\\
-0.04	0.00159999999999982\\
-0.016	0.000256000000000256\\
0.00800000000000001	6.4000000000064e-05\\
0.032	0.00102400000000014\\
0.056	0.00313599999999958\\
0.0800000000000001	0.00640000000000018\\
0.112	0.0125440000000001\\
0.144	0.0207360000000003\\
0.176	0.0309759999999999\\
0.208	0.0432639999999997\\
0.24	0.0575999999999999\\
0.272	0.0739840000000003\\
0.304	0.0924160000000001\\
0.336	0.112896\\
0.368	0.135424\\
0.4	0.16\\
0.432	0.186624\\
0.464	0.215296\\
0.496	0.246016\\
0.528	0.278784\\
0.56	0.3136\\
0.592000000000001	0.350464000000001\\
0.624000000000001	0.389376\\
0.656000000000001	0.430336\\
0.688	0.473344\\
0.728	0.529984\\
0.768	0.589824\\
0.808	0.652864\\
0.848	0.719104\\
0.888	0.788544\\
0.928	0.861184\\
0.968	0.937024\\
1	1\\
1.04	0.9216\\
1.08	0.8464\\
1.12	0.7744\\
1.16	0.7056\\
1.2	0.640000000000001\\
1.24	0.5776\\
1.28	0.518400000000001\\
1.32	0.462400000000001\\
1.36	0.4096\\
1.392	0.369664\\
1.424	0.331776000000001\\
1.456	0.295936\\
1.488	0.262144\\
1.52	0.2304\\
1.552	0.200704\\
1.584	0.173056\\
1.616	0.147456\\
1.648	0.123904\\
1.68	0.1024\\
1.712	0.0829440000000004\\
1.744	0.0655359999999998\\
1.776	0.0501760000000004\\
1.808	0.0368640000000005\\
1.84	0.0255999999999998\\
1.872	0.0163840000000004\\
1.904	0.00921600000000034\\
1.936	0.00409599999999966\\
1.96	0.00159999999999982\\
1.984	0.000256000000000256\\
2.024	0\\
4	0\\
};
\addlegendentry{$u_4^\ast$}

\end{axis}
\end{tikzpicture}%
  \caption{the $2-$semiconcave envelope of the function \\
   $u_4$ defined in \eqref{Example4.5}.}
  \label{Fig4b}
\end{subfigure}
\caption{The function $u_3^\ast$  (resp. $u_4^\ast$) is the smallest reachable target, for $T=1$ (resp. $T=0.5$), bounded from below by $u_3$ (resp. $u_4$).}
\end{figure}

\end{example}

\begin{example}\label{Ex5}
We consider now the two-dimensional case with the Hamiltonian
$$
H(p) := \dfrac{\langle A\, p,p\rangle}{2}, \qquad \text{with} \ 
A=\left(\begin{array}{cc}
2 & 1 \\ 1 & 1
\end{array}\right).
$$
We have computed numerically the image by $S_T^+\circ S_T^-$, with $T=1$, of the function
$$
u_5 (x,y):= \left\{ \begin{array}{ll}
1- |(x,y) - (-1,0)|, & \text{if} \ |(x,y) - (-1,0)| < 1 \\
0.5(1 -  |(x,y) - (1,0)|), & \text{if} \ |(x,y) - (1,0)| < 1 \\
0,   &   \text{else.}
\end{array}\right.
$$
and with $T=0.5$ for the function
$$
u_6 (x,y) = - u_5(x,y).
$$

In Figure \ref{Fig5}, we can see the function $u_5$ at the left and 
its $A^{-1}-$semiconcave envelope $u_5^\ast := S_1^+(S_1^- u_5)$ at the right.
In Figure \ref{Fig6}, we can see the function $u_6$ at the left and 
its $2A^{-1}-$semiconcave envelope $u_6^\ast := S_{0.5}^+(S_{0.5}^- u_6)$ at the right.

\begin{figure}[h]
  \centering
\includegraphics[width=12cm, height=4cm]{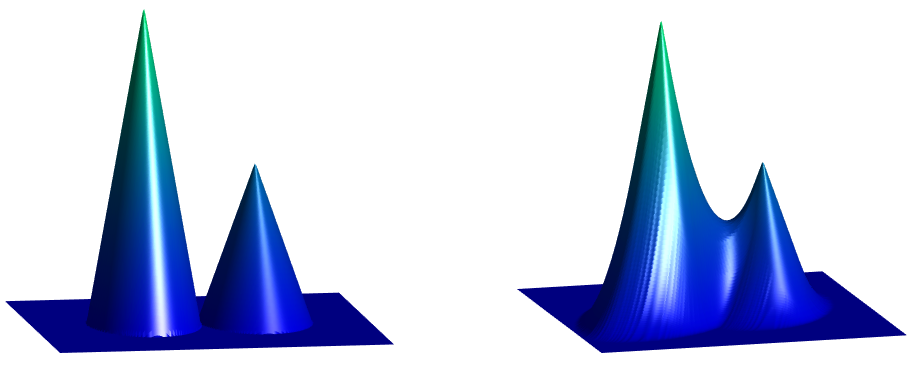}
  \caption{At the left, we see the function $u_5$ from Example \ref{Ex5}. 
  At the right, we see the function $u_5^\ast = S_T^+(S_T^- u_5)$ with $T=1$.}
\label{Fig5}
\end{figure}

\begin{figure}[h]
  \centering
\includegraphics[width=12cm, height=4cm]{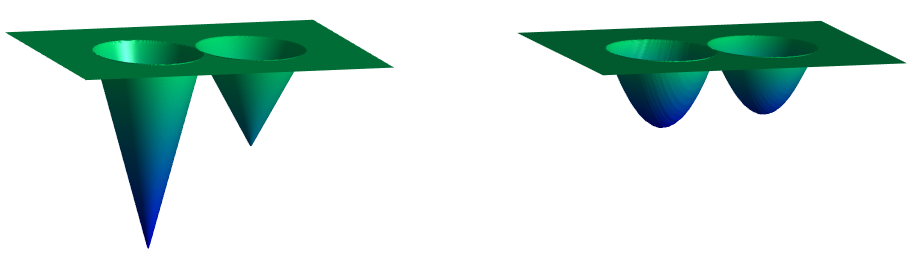}
  \caption{At the left, we see the function $u_6$ from Example \ref{Ex5}. 
  At the right, we see the function $u_6^\ast = S_T^+(S_T^- u_6)$ with $T=0.5$.}
\label{Fig6}
\end{figure}
\end{example}

\section{Forward and backward viscosity solutions}\label{Sec Backward solutions}

We start this section by recalling the definition of viscosity solution for the equation
\begin{equation}\label{HJ eq}
\partial_t u + H(D_x u)=0, \qquad \text{in} \quad [0,T]\times\R^n,
\end{equation}
where $H$ is a continuous function $\R^n\to \R$.

This notion of solution was introduced by Carndall and Lions in \cite{crandall1983viscosity}
and solved the problem of the lack of uniqueness for generalized solutions to the initial-value problem \eqref{HamJac eq}, 
that satisfy the equation \eqref{HJ eq} almost everywhere along with the initial condition $u(0,\cdot)= u_0$.

\begin{definition}\label{Def visc sol HJ}
A uniformly continuous function $u:[0,T]\times\R^n\to \R$
is called a viscosity solution of \eqref{HJ eq} if the following two statements hold:
\begin{enumerate}
\item $u$ is a viscosity subsolution of \eqref{HamJac eq}:
for each $\varphi\in C^\infty ([0,T]\times\R^n)$,
$$
\partial_t \varphi(t_0,x_0) + H(\nabla_x \varphi(t_0,x_0)) \leq 0
$$
whenever $(t_0,x_0)$ is a local maximum of $u-\varphi$.

\item $u$ is a viscosity supersolution of \eqref{HamJac eq}:
for each $\varphi\in C^\infty ([0,T]\times\R^n)$,
$$
\partial_t \varphi(t_0,x_0) + H(\nabla_x \varphi(t_0,x_0)) \geq 0
$$
whenever $(t_0,x_0)$ is a local minimum of $u-\varphi$.
\end{enumerate}
\end{definition}

Throughout the paper we will sometimes refer to viscosity solutions as \emph{forward viscosity solutions},
in contrast with the notion of backward viscosity solution, stated in Definition \ref{Def Backward soluiton}.
In \cite{crandall1983viscosity} (see also \cite{barles1994solutions,cannarsa2004semiconcave,lions1982generalized}),
it is proved the existence and uniqueness of a viscosity solution for the problem \eqref{HamJac eq} for any initial condition $u_0\in\Lip(\R^n)$.
This solution can be obtained by means of the Hopf-Lax formula (see 
\cite{alvarez1999hopf,bardi1984hopf,barles1987uniqueness,lions1982generalized}).
Therefore, the operator $S_T^+$ defined in the introduction can be written as
\begin{equation}\label{Hopf formula}
S_T^+ u_0 (x) = \min_{y\in\R^n} \left[ u_0(y) + T\, L\left(\dfrac{x-y}{T}\right)\right],
\end{equation}
where, the function $L:\R^n\to \R$ is the Legendre transform of $H$, defined as
\begin{equation}\label{Legendre transform}
L(q) := H^\ast (q) = \max_{p\in\R^n} [ q\cdot p - H(p)]. 
\end{equation}
This function corresponds to the Lagrangian in the optimal control problem associated to \eqref{HamJac eq}.
We recall that, under the assumptions \eqref{CondH} on $H$, 
the function $L = H^\ast$ is a convex $C^2$ function satisfying
$$
\lim_{|q|\to\infty} \dfrac{L(q)}{|q|} = +\infty.
$$
See for example Section A.2 in \cite{cannarsa2004semiconcave}. 
We then deduce that the minimum in the Hopf-Lax formula \eqref{Hopf formula} is always attained.

Observe that in the case of a quadratic Hamiltonian of the form \eqref{quadratic Hamiltonian}, an elementary computation gives the Lengendre transform of $H$ as
$$
L(q) = \dfrac{\langle A^{-1}q, q\rangle}{2}.
$$

As announced in the introduction, a key point in our study is the possibility of reversing the direction of time in problem \eqref{HamJac eq}.
This can be done with the notion of \emph{backward viscosity solution} 
(see for example \cite{barron1999regularity}).

\begin{definition}\label{Def Backward soluiton}
A function $w:[0,T]\times\R^n\longrightarrow \R$ is a backward viscosity solution of \eqref{HJ eq} if the function $v$ 
obtained from $w$  by ``reversing the time'', i.e. $v (t,x):=w(T-t,x)$, is a viscosity solution of
$$
\partial_t v- H(D_x v) = 0, \qquad \text{in} \ [0,T]\times\R^n.
$$
\end{definition}

It is clear that a function $w\in C^1([0,T]\times\R^n)$ is a backward viscosity solution if and only if it is a (forward) viscosity solution. 
However, when one deals with non-smooth solutions, both notions of solution are no longer equivalent.
Indeed, viscosity solutions are characterized to be semiconcave, while backward viscosity solutions are semiconvex.
See Proposition \ref{Prop semiconcave} and the Example \ref{Example 1} for an illustration of this phenomenon.
The following characterization of backward viscosity solutions follows immediately from the definition.

\begin{proposition}\label{Prop def backward sol}
A uniformly continuous function $w:[0,T]\times\R^n\to \R$
is a backward viscosity solution to \eqref{HJ eq} if and only if the following two properties hold
\begin{enumerate}
\item for each $\phi\in C^\infty ([0,T]\times\R^n)$,
$$
\partial_t\phi(t_0,x_0) + H(\nabla_x \phi(t_0,x_0)) \geq 0
$$
whenever $(t_0,x_0)$ is a local maximum of $w-\phi$.

\item for each $\phi\in C^\infty ([0,T]\times\R^n)$,
$$
\partial_t\phi(t_0,x_0) + H(\nabla_x \phi(t_0,x_0)) \leq 0
$$
whenever $(t_0,x_0)$ is a local minimum of $w-\phi$.
\end{enumerate}
\end{proposition}

Observe that, if we reverse the inequalities in this proposition,
we obtain exactly the definition of (forward) viscosity solution (Definition \ref{Def visc sol HJ}).
We then deduce that any backward solution $w$ satisfies \eqref{HJ eq} at any point where it is differentiable.
Using the analogous arguments as for the existence and uniqueness of 
(forward) viscosity solutions of \eqref{HamJac eq},
one can prove (see for example \cite{bardi2008optimal,barron1999regularity}) 
existence and uniqueness of a backward viscosity solution for the terminal-value problem
\begin{equation*}
\left\{ \begin{array}{ll}
\partial_t w + H(D_x w) = 0, & \text{in} \ [0,T]\times \R^n, \\
\noalign{\vskip 1.5mm}
w(T,x) = u_T(x), & \text{in} \ \R^n.
\end{array}\right.
\end{equation*}

Hence, for any $T>0$, the nonlinear operator $S_T^-$ defined in Section \ref{Sec Results} is well defined.
Recall that this operator associates, to any $u_T\in\Lip(\R^n)$ the backward viscosity solution of \eqref{Backward HamJac eq} at time $0$.
In addition, $S_T^- u_T$ can be given by the following representation formula, 
which is the analogous to the Hopf-Lax formula \eqref{Hopf formula} for the operator $S_T^+$ (see \cite{barron1999regularity}):
\begin{equation}\label{Backward Hopf formula}
S_T^- u_T (x) = \max_{y\in\R^n} \left[ u_T(y) - T\, L\left(\dfrac{y-x}{T}\right)\right].
\end{equation}

\subsection{Semiconcave and semiconvex functions}
Here we recall the following important property of forward and backward viscosity solutions:

\begin{proposition}\label{Prop semiconcave}
Let $H$ satisfy \eqref{CondH}, $T>0$, and $u_0,u_T\in \Lip(\R^n)$.
Then, for any bounded set $\mathcal{K}\subset \R^n$, 
\begin{enumerate}
\item the function $S^+_T u_0$ defined in \eqref{Hopf formula} is semiconcave in $\mathcal{K}$ with linear modulus;
\item the function $S^-_T u_T$ defined in \eqref{Backward Hopf formula} is semiconvex in $\mathcal{K}$ with linear modulus.
\end{enumerate}
\end{proposition}

See for example Chapter 1 in \cite{cannarsa2004semiconcave} for a proof of this property.
Let us recall the definition of semiconcavity and semiconvexity with linear modulus.

\begin{definition}\label{Def semiconcave semiconvex}
\begin{enumerate}
\item We say that a function $f:\mathcal{K}\subset\R^n\rightarrow\R$ is semiconcave with linear modulus if it is continuous and 
there exists $C\geq 0$ such that
$$
f(x+h) + f(x-h) - 2f(x) \leq C|h|^2, \qquad \text{for all $x,h\in\R^n$, such that} \ [x-h,x+h]\subset \mathcal{K}.
$$
The constant $C$ above is called a semiconcavity constant for $f$ in $\mathcal{K}$.

\item We say that $f$ is semiconvex if the function $g=-f$ is semiconcave.
\end{enumerate}
\end{definition}

Let us illustrate this property with the following example in dimension 1.

\begin{example}\label{Example 1}
We consider the one-dimensional case and the Hamiltonian $H(p) = |p|^2/2$.
We apply the operators $S^+_T$ and $S_T^-$ to the functions
$u_3$ and $u_4$ from Example \ref{Ex4}, with $T=1$ and $=0.5$ respectively.

In Figure \ref{Fig1}, we can see the function $u_3$ represented in both plots by a dotted line.
In the plot at the left, we can see the function $S_T^+u_3$. 
We observe that it is a semiconcave function (roughly speaking, the second derivative is ``bounded from above'').
However, in the plot at the right, we see that the function $S_T^-u_3$ is a semiconvex function (the second derivative is ``bounded from below'').
The same behaviour is observed in Figure \ref{Fig2} for the function $u_4$ from Example \ref{Ex4}.

\begin{figure}[h]
\centering
\begin{subfigure}{.5\textwidth}
  \centering
  \begin{tikzpicture}

\begin{axis}[%
width=2.8in,
height=1.2in,
at={(0.758in,0.481in)},
scale only axis,
xmin=-4,
xmax=7,
ymin=-1,
ymax=0,
axis line style={draw=none},
ticks=none,
legend style={at={(0.97,0.03)}, anchor=south east, legend cell align=left, align=left, draw=white!15!black}
]
\addplot [color=black, dotted, line width=1.1pt]
  table[row sep=crcr]{%
-4	0\\
-2	0\\
-1	-1\\
0	0\\
1	-1\\
2	0\\
4	0\\
};
\addlegendentry{$u_3$}

\addplot [color=black, line width=1.1pt]
  table[row sep=crcr]{%
-4	0\\
-2.504	0\\
-2.496	-0.00400000000000045\\
-1.968	-0.531488\\
-1.936	-0.561952\\
-1.904	-0.591392\\
-1.872	-0.619808\\
-1.84	-0.6472\\
-1.808	-0.673568\\
-1.776	-0.698912\\
-1.744	-0.723232\\
-1.712	-0.746528000000001\\
-1.68	-0.7688\\
-1.648	-0.790048\\
-1.616	-0.810272\\
-1.584	-0.829472\\
-1.552	-0.847648\\
-1.52	-0.8648\\
-1.488	-0.880928\\
-1.456	-0.896032\\
-1.424	-0.910112\\
-1.392	-0.923168\\
-1.36	-0.9352\\
-1.328	-0.946208\\
-1.296	-0.956192\\
-1.264	-0.965152\\
-1.232	-0.973088\\
-1.2	-0.98\\
-1.168	-0.985888\\
-1.136	-0.990752\\
-1.104	-0.994592\\
-1.08	-0.9968\\
-1.056	-0.998432\\
-1.032	-0.999488\\
-1.008	-0.999968\\
-0.984	-0.999872\\
-0.96	-0.9992\\
-0.936	-0.997952\\
-0.912	-0.996128\\
-0.888	-0.993728\\
-0.864	-0.990752\\
-0.832	-0.985888\\
-0.8	-0.98\\
-0.768	-0.973088\\
-0.736	-0.965152\\
-0.704	-0.956192\\
-0.672	-0.946208\\
-0.640000000000001	-0.9352\\
-0.608000000000001	-0.923168\\
-0.576000000000001	-0.910112\\
-0.544	-0.896032\\
-0.512	-0.880928\\
-0.48	-0.8648\\
-0.448	-0.847648\\
-0.416	-0.829472\\
-0.384	-0.810272\\
-0.352	-0.790048\\
-0.32	-0.7688\\
-0.288	-0.746528\\
-0.256	-0.723232\\
-0.224	-0.698912\\
-0.192	-0.673568\\
-0.16	-0.6472\\
-0.128	-0.619808\\
-0.0960000000000001	-0.591392\\
-0.0640000000000001	-0.561952\\
-0.032	-0.531488\\
0	-0.5\\
0.032	-0.531488\\
0.0640000000000001	-0.561952\\
0.0960000000000001	-0.591392\\
0.128	-0.619808\\
0.16	-0.6472\\
0.192	-0.673568\\
0.224	-0.698912\\
0.256	-0.723232\\
0.288	-0.746528\\
0.32	-0.7688\\
0.352	-0.790048\\
0.384	-0.810272\\
0.416	-0.829472\\
0.448	-0.847648\\
0.48	-0.8648\\
0.512	-0.880928\\
0.544	-0.896032\\
0.576000000000001	-0.910112\\
0.608000000000001	-0.923168\\
0.640000000000001	-0.9352\\
0.672	-0.946208\\
0.704	-0.956192\\
0.736	-0.965152\\
0.768	-0.973088\\
0.8	-0.98\\
0.832	-0.985888\\
0.864	-0.990752\\
0.896	-0.994592\\
0.92	-0.9968\\
0.944	-0.998432\\
0.968	-0.999488\\
0.992	-0.999968\\
1.016	-0.999872\\
1.04	-0.9992\\
1.064	-0.997952\\
1.088	-0.996128\\
1.112	-0.993728\\
1.136	-0.990752\\
1.168	-0.985888\\
1.2	-0.98\\
1.232	-0.973088\\
1.264	-0.965152\\
1.296	-0.956192\\
1.328	-0.946208\\
1.36	-0.9352\\
1.392	-0.923168\\
1.424	-0.910112\\
1.456	-0.896032\\
1.488	-0.880928\\
1.52	-0.8648\\
1.552	-0.847648\\
1.584	-0.829472\\
1.616	-0.810272\\
1.648	-0.790048\\
1.68	-0.7688\\
1.712	-0.746528000000001\\
1.744	-0.723232\\
1.776	-0.698912\\
1.808	-0.673568\\
1.84	-0.6472\\
1.872	-0.619808\\
1.904	-0.591392\\
1.936	-0.561952\\
1.968	-0.531488\\
2	-0.5\\
2.496	-0.00400000000000045\\
2.504	0\\
4	0\\
};
\addlegendentry{$S_T^+u_3$}

\end{axis}

\end{tikzpicture}%
  \caption{The function $S_T^+ u_3$ with $T=1$.}
  \label{Fig1a}
\end{subfigure}%
\begin{subfigure}{.5\textwidth}
  \centering
  \begin{tikzpicture}

\begin{axis}[%
width=2.8in,
height=1.2in,
at={(0.758in,0.481in)},
scale only axis,
xmin=-4,
xmax=7,
ymin=-1,
ymax=0,
axis line style={draw=none},
ticks=none,
legend style={at={(0.97,0.03)}, anchor=south east, legend cell align=left, align=left, draw=white!15!black}
]
\addplot [color=black, dotted, line width=1.1pt]
  table[row sep=crcr]{%
-4	0\\
-2	0\\
-1	-1\\
0	0\\
1	-1\\
2	0\\
4	0\\
};
\addlegendentry{$u_3$}

\addplot [color=black, line width=1.1pt]
  table[row sep=crcr]{%
-4	0\\
-1.976	-0.000288000000000288\\
-1.952	-0.00115200000000026\\
-1.928	-0.00259199999999993\\
-1.904	-0.00460800000000017\\
-1.88	-0.0072000000000001\\
-1.856	-0.0103679999999997\\
-1.824	-0.0154880000000004\\
-1.792	-0.0216320000000003\\
-1.76	-0.0288000000000004\\
-1.728	-0.0369919999999997\\
-1.696	-0.046208\\
-1.664	-0.0564479999999996\\
-1.632	-0.0677120000000002\\
-1.6	-0.0800000000000001\\
-1.568	-0.0933120000000001\\
-1.536	-0.107648\\
-1.504	-0.123008\\
-1.472	-0.139392\\
-1.44	-0.1568\\
-1.408	-0.175232\\
-1.376	-0.194688\\
-1.344	-0.215168\\
-1.312	-0.236672\\
-1.28	-0.2592\\
-1.248	-0.282752\\
-1.216	-0.307328\\
-1.184	-0.332928\\
-1.152	-0.359552\\
-1.12	-0.3872\\
-1.088	-0.415872\\
-1.056	-0.445568\\
-1.024	-0.476288\\
-1	-0.5\\
-0.968	-0.468512\\
-0.936	-0.438048\\
-0.904	-0.408608\\
-0.872	-0.380192\\
-0.84	-0.3528\\
-0.808	-0.326432\\
-0.776	-0.301088\\
-0.744	-0.276768\\
-0.712	-0.253471999999999\\
-0.68	-0.231199999999999\\
-0.648000000000001	-0.209952\\
-0.616000000000001	-0.189728\\
-0.584000000000001	-0.170528\\
-0.552	-0.152352\\
-0.52	-0.1352\\
-0.488	-0.119072\\
-0.456	-0.103968\\
-0.424	-0.0898880000000002\\
-0.392	-0.0768319999999996\\
-0.36	-0.0648\\
-0.328	-0.0537919999999996\\
-0.296	-0.0438080000000003\\
-0.264	-0.0348480000000002\\
-0.232	-0.0269120000000003\\
-0.2	-0.0199999999999996\\
-0.168	-0.0141119999999999\\
-0.136	-0.00924800000000037\\
-0.104	-0.00540800000000008\\
-0.0800000000000001	-0.00319999999999965\\
-0.056	-0.00156799999999979\\
-0.032	-0.000511999999999624\\
-0.00800000000000001	-3.2000000000032e-05\\
0.016	-0.000128000000000128\\
0.04	-0.000799999999999912\\
0.0640000000000001	-0.00204800000000027\\
0.0880000000000001	-0.00387200000000032\\
0.112	-0.00627200000000006\\
0.136	-0.00924800000000037\\
0.168	-0.0141119999999999\\
0.2	-0.0199999999999996\\
0.232	-0.0269120000000003\\
0.264	-0.0348480000000002\\
0.296	-0.0438080000000003\\
0.328	-0.0537919999999996\\
0.36	-0.0648\\
0.392	-0.0768319999999996\\
0.424	-0.0898880000000002\\
0.456	-0.103968\\
0.488	-0.119072\\
0.52	-0.1352\\
0.552	-0.152352\\
0.584000000000001	-0.170528\\
0.616000000000001	-0.189728\\
0.648000000000001	-0.209952\\
0.68	-0.231199999999999\\
0.712	-0.253471999999999\\
0.744	-0.276768\\
0.776	-0.301088\\
0.808	-0.326432\\
0.84	-0.3528\\
0.872	-0.380192\\
0.904	-0.408608\\
0.936	-0.438048\\
0.968	-0.468512\\
1	-0.5\\
1.032	-0.468512\\
1.064	-0.438048\\
1.096	-0.408608\\
1.128	-0.380192\\
1.16	-0.3528\\
1.192	-0.326432\\
1.224	-0.301088\\
1.256	-0.276768000000001\\
1.288	-0.253472\\
1.32	-0.2312\\
1.352	-0.209952\\
1.384	-0.189728\\
1.416	-0.170528\\
1.448	-0.152352\\
1.48	-0.1352\\
1.512	-0.119072\\
1.544	-0.103968\\
1.576	-0.0898880000000002\\
1.608	-0.0768319999999996\\
1.64	-0.0648\\
1.672	-0.0537919999999996\\
1.704	-0.0438080000000003\\
1.736	-0.0348480000000002\\
1.768	-0.0269120000000003\\
1.8	-0.0200000000000005\\
1.832	-0.0141119999999999\\
1.864	-0.00924800000000037\\
1.896	-0.00540800000000008\\
1.92	-0.00319999999999965\\
1.944	-0.00156799999999979\\
1.968	-0.000511999999999624\\
1.992	-3.2000000000032e-05\\
2.088	0\\
4	0\\
};
\addlegendentry{$S_T^- u_3$}

\end{axis}
\end{tikzpicture}%
  \caption{The function $S_T^- u_3$ with $T=1$.}
  \label{Fig1b}
\end{subfigure}
\caption{The dotted line represents the function $u_3$ defined in \eqref{Example4}.}
\label{Fig1}
\end{figure}
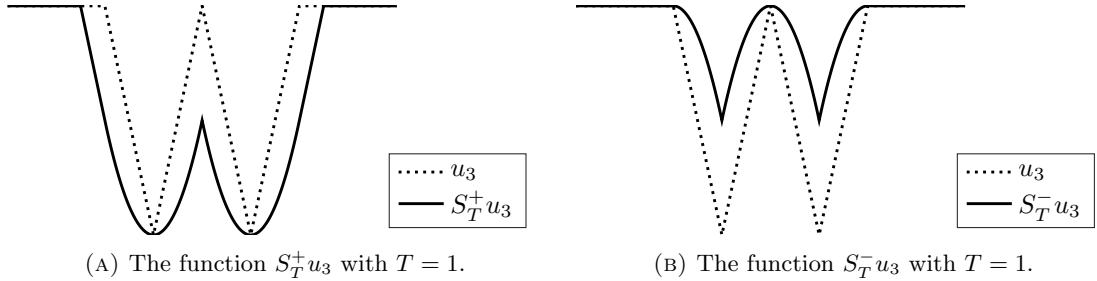

\begin{figure}[h]
\centering
\begin{subfigure}{.5\textwidth}
  \centering
\begin{tikzpicture}

\begin{axis}[%
width=2.2in,
height=1.3in,
at={(0.758in,0.481in)},
scale only axis,
xmin=-4,
xmax=4,
ymin=-1,
ymax=1,
axis line style={draw=none},
ticks=none,
legend style={at={(0.97,0.03)}, anchor=south east, legend cell align=left, align=left, draw=white!15!black}
]
\addplot [color=black, dotted, line width=1.1pt]
  table[row sep=crcr]{%
-4	-0\\
-1.504	-0\\
-1.496	0.00800000000000001\\
-1	1\\
0	-1\\
1	1\\
1.496	0.00800000000000001\\
1.504	-0\\
4	-0\\
};
\addlegendentry{$u_4$}

\addplot [color=black, line width=1.1pt]
  table[row sep=crcr]{%
-4	0\\
-1.48	0.000575999999999688\\
-1.456	0.00230399999999964\\
-1.432	0.00518399999999986\\
-1.4	0.0108160000000002\\
-1.368	0.0184959999999998\\
-1.336	0.0282239999999998\\
-1.304	0.04\\
-1.272	0.0538240000000005\\
-1.24	0.0696960000000004\\
-1.208	0.0876160000000006\\
-1.176	0.107584\\
-1.144	0.1296\\
-1.112	0.153664\\
-1.088	0.173056\\
-1.08	0.16\\
-0.96	-0.0784000000000002\\
-0.92	-0.1536\\
-0.88	-0.2256\\
-0.84	-0.2944\\
-0.8	-0.36\\
-0.76	-0.422400000000001\\
-0.72	-0.4816\\
-0.68	-0.5376\\
-0.640000000000001	-0.5904\\
-0.608000000000001	-0.630336\\
-0.576000000000001	-0.668224\\
-0.544	-0.704064\\
-0.512	-0.737856\\
-0.48	-0.769600000000001\\
-0.448	-0.799296\\
-0.416	-0.826944\\
-0.384	-0.852544\\
-0.352	-0.876096\\
-0.32	-0.8976\\
-0.288	-0.917056000000001\\
-0.256	-0.934464\\
-0.224	-0.949824\\
-0.192	-0.963136\\
-0.16	-0.9744\\
-0.128	-0.983616\\
-0.0960000000000001	-0.990784\\
-0.0640000000000001	-0.995904\\
-0.04	-0.9984\\
-0.016	-0.999744\\
0.00800000000000001	-0.999936\\
0.032	-0.998976\\
0.056	-0.996864\\
0.0800000000000001	-0.9936\\
0.112	-0.987456\\
0.144	-0.979264\\
0.176	-0.969024\\
0.208	-0.956736\\
0.24	-0.9424\\
0.272	-0.926016000000001\\
0.304	-0.907584\\
0.336	-0.887104\\
0.368	-0.864576\\
0.4	-0.84\\
0.432	-0.813376\\
0.464	-0.784704000000001\\
0.496	-0.753984\\
0.528	-0.721216\\
0.56	-0.6864\\
0.592000000000001	-0.649536\\
0.624000000000001	-0.610624\\
0.656000000000001	-0.569664\\
0.688	-0.526656000000001\\
0.728	-0.470016\\
0.768	-0.410176000000001\\
0.808	-0.347136\\
0.848	-0.280896\\
0.888	-0.211456\\
0.928	-0.138816\\
0.968	-0.0629759999999999\\
1.016	0.032\\
1.08	0.16\\
1.088	0.173056\\
1.12	0.147456\\
1.152	0.123904\\
1.184	0.1024\\
1.216	0.0829440000000004\\
1.248	0.0655359999999998\\
1.28	0.0501760000000004\\
1.312	0.0368640000000005\\
1.344	0.0255999999999998\\
1.376	0.0163840000000004\\
1.408	0.00921600000000034\\
1.44	0.00409599999999966\\
1.464	0.00159999999999982\\
1.488	0.000256000000000256\\
1.528	0\\
4	0\\
};
\addlegendentry{$S_T^+u_4$}

\end{axis}
\end{tikzpicture}%
\caption{The function $S_T^+ u_4$ with $T=0.5$.}
  \label{Fig2a}
\end{subfigure}%
\begin{subfigure}{.5\textwidth}
  \centering
\begin{tikzpicture}

\begin{axis}[%
width=2.2in,
height=1.3in,
at={(0.758in,0.481in)},
scale only axis,
xmin=-4,
xmax=4,
ymin=-1,
ymax=1,
axis line style={draw=none},
ticks=none,
legend style={at={(0.97,0.03)}, anchor=south east, legend cell align=left, align=left, draw=white!15!black}
]
\addplot [color=black, dotted, line width=1.1pt]
  table[row sep=crcr]{%
-4	-0\\
-1.504	-0\\
-1.496	0.00800000000000001\\
-1	1\\
0	-1\\
1	1\\
1.496	0.00800000000000001\\
1.504	-0\\
4	-0\\
};
\addlegendentry{$u_4$}

\addplot [color=black, line width=1.1pt]
  table[row sep=crcr]{%
-4	-0\\
-2	-0\\
-1.96	0.0784000000000002\\
-1.92	0.1536\\
-1.88	0.2256\\
-1.84	0.2944\\
-1.8	0.36\\
-1.76	0.422400000000001\\
-1.72	0.4816\\
-1.68	0.5376\\
-1.64	0.5904\\
-1.608	0.630336\\
-1.576	0.668224\\
-1.544	0.704064\\
-1.512	0.737856\\
-1.48	0.769600000000001\\
-1.448	0.799296\\
-1.416	0.826944\\
-1.384	0.852544\\
-1.352	0.876096\\
-1.32	0.8976\\
-1.288	0.917056000000001\\
-1.256	0.934464\\
-1.224	0.949824\\
-1.192	0.963136\\
-1.16	0.9744\\
-1.128	0.983616\\
-1.096	0.990784\\
-1.064	0.995904\\
-1.04	0.9984\\
-1.016	0.999744\\
-0.992	0.999936\\
-0.968	0.998976\\
-0.944	0.996864\\
-0.92	0.9936\\
-0.888	0.987456\\
-0.856	0.979264\\
-0.824	0.969024\\
-0.792	0.956736\\
-0.76	0.9424\\
-0.728	0.926016\\
-0.696	0.907584\\
-0.664000000000001	0.887104\\
-0.632000000000001	0.864576\\
-0.600000000000001	0.84\\
-0.568000000000001	0.813376\\
-0.536	0.784704000000001\\
-0.504	0.753984\\
-0.472	0.721216\\
-0.44	0.6864\\
-0.408	0.649536\\
-0.376	0.610624\\
-0.344	0.569664\\
-0.312	0.526656\\
-0.272	0.470015999999999\\
-0.232	0.410176\\
-0.192	0.347136\\
-0.152	0.280896\\
-0.112	0.211456\\
-0.0720000000000001	0.138816\\
-0.032	0.0629759999999999\\
0	0\\
0.04	0.0784000000000002\\
0.0800000000000001	0.1536\\
0.12	0.2256\\
0.16	0.2944\\
0.2	0.359999999999999\\
0.24	0.4224\\
0.28	0.481599999999999\\
0.32	0.5376\\
0.36	0.5904\\
0.392	0.630336\\
0.424	0.668224\\
0.456	0.704064\\
0.488	0.737856\\
0.52	0.769600000000001\\
0.552	0.799296\\
0.584000000000001	0.826944\\
0.616000000000001	0.852544\\
0.648000000000001	0.876096\\
0.68	0.8976\\
0.712	0.917056\\
0.744	0.934464\\
0.776	0.949824\\
0.808	0.963136\\
0.84	0.9744\\
0.872	0.983616\\
0.904	0.990784\\
0.936	0.995904\\
0.96	0.9984\\
0.984	0.999744\\
1.008	0.999936\\
1.032	0.998976\\
1.056	0.996864\\
1.08	0.9936\\
1.112	0.987456\\
1.144	0.979264\\
1.176	0.969024\\
1.208	0.956736\\
1.24	0.9424\\
1.272	0.926016000000001\\
1.304	0.907584\\
1.336	0.887104\\
1.368	0.864576\\
1.4	0.84\\
1.432	0.813376\\
1.464	0.784704000000001\\
1.496	0.753984\\
1.528	0.721216\\
1.56	0.6864\\
1.592	0.649536\\
1.624	0.610624\\
1.656	0.569664\\
1.688	0.526656\\
1.728	0.470016\\
1.768	0.410176000000001\\
1.808	0.347136\\
1.848	0.280896\\
1.888	0.211456\\
1.928	0.138816\\
1.968	0.0629759999999999\\
2	0\\
4	-0\\
};
\addlegendentry{$S_T^- u_4$}

\end{axis}

\end{tikzpicture}%
\caption{The function $S_T^- u_4$ with $T=0.5$}
  \label{Fig2b}
\end{subfigure}
\caption{The dotted line represents the function $u_4$ defined in \eqref{Example4.5}.}
\label{Fig2}
\end{figure}

\end{example}

The following proposition, 
whose proof can be found in Chapter 1 in \cite{cannarsa2004semiconcave},
gives an interesting characterization of semiconcave functions with linear modulus.
Combining this characterization with Theorem \ref{Thm Oberman convex} in subsection \ref{Sec semiconcave 1.2}, 
it can be easily deduced that any viscosity solution to the differential inequality 
\begin{equation*}
\lambda_n \left[ D^2 v - \dfrac{1}{T}A^{-1}\right] \leq 0,
\end{equation*}
where $A$ is a positive definite matrix, are semiconcave with a linear modulus.

\begin{proposition}\label{Prop semiconcave characterization}
Given a function $f:\R^n\to\R$ and a constant $C\geq 0$,
the following properties are equivalent:
\begin{enumerate}
\item $f$ is semiconcave with linear modulus and constant $C$;
\item the function $x\mapsto f(x) - \dfrac{C}{2}|x|^2$ is concave;
\item there exist two functions $f_1,f_2: \R^n\to \R$ such that $f=f_1+f_2$,
$f_1$ is concave, $f_2\in C^2(\R^n)$ and satisfies $\|D^2f_2\|_\infty \leq C$.
\end{enumerate} 
\end{proposition}

\subsection{Proof of Theorem \ref{Thm 1st reach cond}}
This result is a consequence of Proposition \ref{Prop backward stability} below, which ensures that for any initial condition, the process of taking alternatively forward and backward viscosity solutions stabilizes after the first step.

\begin{proposition}\label{Prop backward stability}
Let $H$ satisfy \eqref{CondH}, $u_0\in \Lip(\R^n)$ and $T>0$.
Set the function
\begin{equation*}
\tilde{u}_0 (x) := S_T^-  (S_T^+ u_0) (x), \qquad \text{for} \ x\in\R^n.
\end{equation*}
Then it holds
\begin{equation}\label{backward stab conclusions}
S^+_T u_0 = S^+_T \tilde{u}_0, \qquad \text{and} \qquad u_0 (x) \geq \tilde{u}_0(x), \quad \text{for all} \ x\in \R^n.
\end{equation}
\end{proposition}

The following diagram illustrates this property, that was already proved in \cite{barron1999regularity} for a more general setting.
We have included here the proof for completeness.

\[
\begin{tikzcd}
 u_0 \arrow[bend left]{rrrd}{S_T^+}  \\
S_T^-  (S_T^+ u_0)= \tilde{u}_0 \arrow[bend left,swap]{rrr}{S_T^+} &  &  & S_T^+ u_0 = u_T \arrow[bend left]{lll}{S_T^-}
\end{tikzcd}
\]

\begin{proof}
By formula \eqref{Hopf formula} we have, for any $x\in \R^n$,
$$
S_T^+ u_0 (x) \leq u_0(y) + T \, L\left( \dfrac{x-y}{T}\right), \quad \text{for all} \  y\in \R^n.
$$
By the arbitrariness of $x$, we deduce
$$
u_0(y) \geq S_T^+ u_0 (x) - T \, L\left( \dfrac{x-y}{T}\right), \quad \text{for all} \ x,y\in\R^n.
$$
Now, for any fixed $y\in\R^n$, taking the maximum over $x$ in the right hand side of the above inequality,
and in view of formula \eqref{Backward Hopf formula}, we obtain $u_0(y) \geq S^-_T (S_T^+ u_0) (y)$
for any $y\in\R^n$.
The inequality in \eqref{backward stab conclusions} is then proved.

Now, we use the comparison principle for the viscosity solutions, which can be deduced directly from \eqref{Hopf formula} (see also \cite{lions1982generalized} for more details). We obtain
\begin{equation*}
S_T^+ u_0(x) \geq S_T^+ \tilde{u}_0 (x), \quad
\text{for all} \ x\in \R^n.
\end{equation*}

For the reversed inequality, fix $x_0\in \R^n$ 
and let $y^\ast$ be a minimizer in the right hand side of \eqref{Hopf formula}.
We have
\begin{equation}\label{equlity proof lemma}
S^+_T \tilde{u}_0 (x_0) = \tilde{u}_0 (y^\ast) + T\, L\left( \dfrac{x_0-y^\ast}{T} \right).
\end{equation}
Then, from the definition of $\tilde{u}_0$ and formula  \eqref{Backward Hopf formula}, it follows
$$
\tilde{u}_0 (y^\ast) = \max_{x\in \R^n} 
\left[ S_T^+ u_0(x) - T\, L\left( \dfrac{x-y^\ast}{T}\right) \right] 
\geq S_T^+ u_0(x_0) - T\, L\left( \dfrac{x_0-y^\ast}{T}\right).
$$
And combining this inequality with \eqref{equlity proof lemma}, we obtain
$$
S_T^+ u_0(x_0) \leq \tilde{u}_0 (y^\ast) + T\, L\left(\dfrac{x_0-y^\ast}{T} \right) = S_T^+ \tilde{u}_0 (x_0).
$$ 
\end{proof}

Using the analogous arguments as in the previous proof, based on the formulas  \eqref{Hopf formula} and \eqref{Backward Hopf formula},
we can obtain the following similar result.
In this case, we start with a terminal condition $u_T\in\Lip(\R^n)$ and apply first the operator $S_T^-$,
and then the operator $S_T^+$.
This result will be useful in Section \ref{Sec semiconcave} for the proof of Theorem \ref{Thm semiconcave envelope}.

\begin{proposition}\label{Prop forward stability}
Let $H$ satisfy \eqref{CondH}, $u_T\in \Lip(\R^n)$ and $T>0$.
Set the function
\begin{equation*}
u_T^\ast (x) := S_T^+  (S_T^- u_T) (x), \qquad \text{for} \ x\in\R^n.
\end{equation*}
Then it holds
\begin{equation}\label{forward stab conclusions}
S^-_T u_T = S^-_T u_T^\ast, \qquad \text{and} \qquad u_T (x) \leq u^\ast_T(x), \quad \text{for all} \ x\in \R^n.
\end{equation}
\end{proposition}

We conclude this section with the proof of Theorem \ref{Thm 1st reach cond}.

\begin{proof}[Proof of Theorem \ref{Thm 1st reach cond}]
The conclusion follows immediately from Proposition \ref{Prop backward stability}.
Indeed, if $I_T(u_T)\neq \emptyset$, then there exists $u_0$ such that $S^+_T u_0 = u_T$.
Now, setting the function $\tilde{u}_0 = S_T^-(S_T^+ u_0)$, after Proposition \ref{Prop backward stability} we obtain
$$S_T^+ (S_T^- u_T) = S_T^+(S_T^- (S_T^+ u_0)) = S_T^+ \tilde{u}_0 = S_T^+ u_0 = u_T.$$
Reversely, if $u_T$ satisfies $S_T^+(S_T^- u_T) = u_T$, we have $S_T^- u_T \in I_T(u_T)$. And the proof is concluded.
\end{proof}

\section{Initial data construction}\label{Sec IniDataConstruction}

Our goal in this section is to prove Theorems \ref{Thm IniData Identif} and \ref{Thm X_I charac}
that, for a time horizon $T>0$ and a reachable target $u_T$, 
give a characterization of the set $I_T(u_T)$ defined in \eqref{Initial conditions}.
For the proofs of these two theorems, we will use the following well-known 
property of viscosity solutions and Hopf-Lax formula.
The proof of this property can be found, for example, in \cite{cannarsa2004semiconcave}.

\begin{proposition}\label{Prop unique min Hopf}
Let $H$ satisfy  \eqref{CondH}, $u_0\in \Lip (\R^n)$ and $T>0$.
For any fixed $x\in \R^n$, the minimizer in the right hand side of \eqref{Hopf formula} 
is unique if and only if the function $S_T^+ u_0(\cdot)$ is differentiable at $x$.
Moreover, in this case, the minimizer is given by $y^\ast = x-T\, H_p(\nabla S^+_T u_0(x))$.
\end{proposition}

Let us now proceed with the proof of Theorem \ref{Thm IniData Identif}:

\begin{proof}[Proof of Theorem \ref{Thm IniData Identif}]
First of all note that, after Theorem \ref{Thm 1st reach cond}, $I_T(u_T)\neq \emptyset$
implies $S_T^+ \tilde{u}_0 =u_T$.

\textit{Step 1: \eqref{Thm IniData Identif u0 geq uast} implies \eqref{Thm IniData Identif u_0 in I_T}.}
Let $u_0\in\Lip (\R^n)$ be any function satisfying the condition \eqref{Thm IniData Identif u0 geq uast} in the enunciate.
Since $u_0\geq \tilde{u}_0$, it follows from the comparison principle (see \cite{lions1982generalized}) that
$$
S_T^+ u_0(x) \geq S_T^+ \tilde{u}_0 (x) = u_T (x), \qquad \text{for all} \ x\in\R^n.
$$

Let us prove the reversed inequality.
Fix any $x_0\in\R^n$ such that $u_T$ is differentiable at $x_0$, and consider $y_0 = x_0 -T\,H_p(\nabla u_T(x_0))$.
Note that $y_0\in X_T(u_T)$ and then, by the assumption \eqref{Thm IniData Identif u0 geq uast},
we have $u_0 (y_0) = \tilde{u}_0(y_0)$.

Since $u_T = S_T^+ \tilde{u}_0$ is differentiable at $x_0$,
we know by Proposition \ref{Prop unique min Hopf} that $y_0$ is the unique minimizer in the Hopf-Lax formula
\eqref{Hopf formula} applied to $\tilde{u}_0$.
Hence, 
$$
u_T(x_0) = \tilde{u}_0 (y_0) +T\, L\left( \dfrac{x_0-y_0}{T}\right).
$$

Applying now  the formula \eqref{Hopf formula} to the function $u_0$,
and using $u_0(y_0) = \tilde{u}_0 (y_0)$, we obtain
$$
S_T^+ u_0(x_0) \leq u_0(y_0) +T\, L\left( \dfrac{x_0-y_0}{T}\right) 
= \tilde{u}_0(y_0) + T\, L\left( \dfrac{x_0-y_0}{T}\right) = u_T(x_0).
$$

We have proved that $S_T^+ u_0 (x) \leq u_T(x)$ for all $x$ where $u_T$ is differentiable.
Since $u_T$ is Lipschitz continuous, by Rademacher's Theorem, $u_T$ is differentiable almost everywhere in $\R^n$,
and then, by the continuity of viscosity solutions, we conclude that $S_T^+ u_0 (x) \leq u_T(x)$ for all $x\in\R^n$.
 
\textit{Step 2: \eqref{Thm IniData Identif u_0 in I_T} implies \eqref{Thm IniData Identif u0 geq uast}.}
Let $u_0\in I_T(u_T)$.
Recall that, by Lemma \ref{Prop backward stability}, we have 
\begin{equation}\label{step 2 IniData}
S_T^+ u_0 = S_T^+ \tilde{u}_0 =u_T \qquad \text{and} \qquad  
u_0 (x) \geq \tilde{u}_0 (x), \ \text{for all} \ x\in \R^n.
\end{equation}

Let us prove that $u_0\in I_T(u_T)$ also implies $u_0(x) = \tilde{u}_0(x)$ for any $x\in X_T(u_T)$.
Fix any $x_0\in X_T(u_T)$.
By the definition of $X_T(u_T)$, there exists $z_0\in\R^n$ such that $u_T$ is differentiable at $z_0$ and
$x_0 = z_0 - T\, H_p(\nabla u_T(z_0))$.

By Proposition \ref{Prop unique min Hopf}, together with \eqref{step 2 IniData}, it follows that $x_0$ is the unique minimizer in Hopf-Lax formula \eqref{Hopf formula} with $x=z_0$ and initial data $u_0$ and $\tilde{u}_0$.
Hence, we deduce that
$$
u_T (z_0) = u_0(x_0) + T\, L\left( \dfrac{z_0-x_0}{T}\right) 
= \tilde{u}_0 (x_0) + T\, L\left( \dfrac{z_0-x_0}{T}\right),
$$
which implies $u_0(x_0) = \tilde{u}_0(x_0)$.
\end{proof}

We now prove a result that characterizes the set $X_T(u_T)$, introduced in Theorem \ref{Thm X_I charac}, 
as the set of points $x_0$ for which there exists a function of the form
$$
-T\, L \left( b + \dfrac{x_0 - x}{T}\right) + c, \qquad \text{with} \ b\in\R^n, \ c\in\R,
$$
touching $\tilde{u}_0$ from below at $x_0$.
Note that Theorem \ref{Thm X_I charac} in Section \ref{Sec Results} is a particular case of the following proposition.
However, the assumption \eqref{quadratic Hamiltonian} in Theorem \ref{Thm X_I charac} allows an important simplification of the statement.

\begin{proposition}\label{Prop X_I charac general}
Let $H$ satisfy \eqref{CondH} and $T>0$.
Let $u_T\in \Lip(\R^n)$ be such that $I_T(u_T)\neq \emptyset$ and take any function $u_0\in I_T(u_T)$.
Then, for any $x_0\in\R^n$, the following two statements are equivalent:
\begin{enumerate}
\item $x_0\in X_T(u_T)$;
\item There exist $p_0\in D^- u_0 (x_0)$ and $c_0\in \R$ such that 
$$
u_0 (x_0) = - T \, L\left( H_p(p_0)\right) + c_0 \qquad \text{and}
$$
$$
u_0 (x) > - T \, L\left( H_p(p_0) + \dfrac{x_0-x}{T}\right) + c_0, \qquad \forall x\in \R^n\setminus \{x_0\}.
$$
\end{enumerate}
\end{proposition}

Here we recall the definition of subdifferential of a function $u_0\in\Lip(\R^n)$:
\begin{equation}\label{Def subdiff}
D^-u_0 (x_0) = \left\{ p\in\R^n\, ; \ \exists \varphi\in C^1(\R^n), \ \nabla\varphi(x_0)=p, \ u_0-\varphi \geq 0, \ (u_0-\varphi)(x_0)=0\right\}.
\end{equation}

\begin{proof}
Let $u_0$ be any initial condition in $I_T(u_T)$,
i.e. $u_0$ satisfies $S_T^+ u_0 =u_T$. 
Let us recall the definition of $X_T(u_T)$ from Theorem \ref{Thm IniData Identif}:
$$
X_T(u_T) = \left\{ z-T\, H_p (\nabla_x u_T(z))\, ; \ \forall z\in\R^n\ \text{such that}\ u_T \ \text{is differentiable at z}\right\}.
$$

Consider any $x_0\in X_T (u_T)$. 
There exists $z_0\in\R^n$ such that $u_T$ is differentiable at $z_0$ and 
\begin{equation}\label{identity Prop proof}
x_0 = z_0-T\, H_p(\nabla_x u_T(z_0)).
\end{equation}
From Proposition \ref{Prop unique min Hopf}, we deduce that $x_0$ is the unique minimizer in
the Hopf-Lax formula \eqref{Hopf formula} with $x=z_0$.
Therefore, we have
$$
u_0 (x_0) = u_T(z_0) - T\, L\left(\dfrac{z_0-x_0}{T}\right)
$$
and
$$
u_0 (x) > u_T(z_0) - T\, L\left(\dfrac{z_0-x}{T}\right), \quad \text{for all} \ x\in\R^n\setminus\{x_0\}.
$$
Taking into account these two relations and the identity \eqref{identity Prop proof},
the statement (ii) follows with $c_0=u_T(z_0)$ and $p_0= \nabla_x u_T(z_0)$.
We observe that, in view of the definition of subdifferential in \eqref{Def subdiff},
we have $L_q(H_p(p_0))\in D^- u_0(x_0)$, and using the properties of the Legendre transform (see Property \ref{Prop Legendre} below),
it follows that $p_0\in D^- u_0(x_0)$.

Conversely, if statement (ii) holds, we see that $x_0$ is the unique minimizer in
formula \eqref{Hopf formula} with $x = z_0$, where $z_0=x_0 + T\, H_p(p_0)$.
From Proposition \ref{Prop unique min Hopf}, it follows that $u_T$ is differentiable at 
$z_0$ and $\nabla_x u_T(z_0) =p_0$. 
Then $x_0\in X_T(u_T)$.
\end{proof}

Here we recall the following elementary property of the Legendre transform that has been used in the previous proof. 
See for example Section A.2 in
\cite{cannarsa2004semiconcave} for the proof of this property.

\begin{property}\label{Prop Legendre}
Under the assumptions \eqref{CondH} on $H$, the function $L$ defined by
$$
L(q) = \max_{p\in\R^n} [ q\cdot p - H(p)]
$$
is a strictly convex $C^2$ function, and its gradient $L_q$ is a $C^1$ diffeomorphism from $\R^n$ to $\R^n$ satisfying
$$
H_p (p) = (L_q)^{-1} (p) \quad \text{and} \quad
\left[ H_{pp} (p)\right]^{-1} = L_{qq}\left( H_p(p)\right),\quad \text{for all} \ p\in\R^n.
$$
\end{property}

We end the section with the proof of Theorem \ref{Thm X_I charac}.

\begin{proof}[Proof of Theorem \ref{Thm X_I charac}]
The result is a particular case of Proposition \ref{Prop X_I charac general}.
Note that in this case $H$ is given by \eqref{quadratic Hamiltonian}.
Hence, we have
\begin{equation}\label{H_p and L(q) quadratic}
H_p (p) = A\, p \qquad \text{and} \qquad L(q) = H^\ast (q) = \dfrac{\langle A^{-1}q, q\rangle}{2}. 
\end{equation}

Consider any $u_0\in I_T(u_T)$.
By Proposition \ref{Prop X_I charac general}, if $x_0\in X_T(u_T)$,
then we have
$$
u_0 (x_0) = -T\, L(H_p(p_0))+c_0,
$$
and
$$
u_0 (x) > -T\, L\left( H_p(p_0) + \dfrac{x_0-x}{T}\right) + c_0 ,\quad \forall x\in \R^n\setminus \{x_0\},
$$
for some $p_0\in D^- u_0 (x_0)$ and $c_0\in \R$.
The identities in \eqref{H_p and L(q) quadratic} and a simple computation give
$$
u_0 (x_0) = -T\, \dfrac{\langle A\, p_0, p_0\rangle}{2} + c_0,
$$
and
$$
u_0 (x) > -\dfrac{\langle A^{-1}(x_0-x),\, x_0-x\rangle}{2T} - \langle p_0,\, x_0-x \rangle - T\, \dfrac{\langle A\, p_0,\, p_0\rangle}{2} + c_0,
\quad \forall x\in \R^n\setminus \{x_0\}.
$$
The statement \eqref{Thm X_I charac ii} in Theorem \ref{Thm X_I charac} then holds with
$$
b := \dfrac{1}{T} A^{-1} x_0 + p_0 \qquad \text{and} \qquad 
c := - \dfrac{\langle A^{-1}x_0,x_0\rangle}{2T} - \langle p_0,\, x_0\rangle 
- T\dfrac{\langle A\, p_0,\, p_0\rangle}{2} + c_0 .
$$

Now, let $x_0\in\R^n$ be such that condition \eqref{Thm X_I charac ii} in Theorem \ref{Thm X_I charac} holds for some
$b\in\R^n$ and $c\in \R$.
In view of the definition of subdifferential \eqref{Def subdiff}, we see that
$$
p_0:= -\dfrac{1}{T} A^{-1} x_0 + b\in D^- u_0 (x_0).
$$
With this choice of $p_0$ and taking $c_0 := u_0(x_0) + T\, \dfrac{\langle A\, p_0, p_0\rangle}{2}$,
it is not difficult to verify that condition (ii) in Proposition \ref{Prop X_I charac general} holds, and then $x_0\in X_T(u_T)$.

\end{proof}

\section{Semiconcave envelopes}\label{Sec semiconcave}

In this section, we study the following fully nonlinear obstacle problem
\begin{equation}\label{obstacle pbm Sec 4}
\min \left\{ v - u_T, \, 
- \lambda_n \left[ D^2 v- \frac{\left[ H_{pp} (D v) \right]^{-1}}{T}\right] \right\} = 0,
\end{equation}
where $H$ is a function satisfying \eqref{CondH}, and $u_T\in\Lip (\R^n)$ and $T>0$ are given.

We first prove Theorem \ref{Thm semiconcave envelope},
which ensures that, for the one-dimensional case and when $H$ is quadratic (i.e. given by \eqref{quadratic Hamiltonian}), the function
$u_T^\ast := S_T^+ (S_T^- u_T)$ is a viscosity solution of \eqref{obstacle pbm Sec 4}. 
This function is obtained by solving the terminal-value problem \eqref{Backward HamJac eq}
and then the initial-value problem \eqref{HamJac eq} with initial condition $S^-_T u_T$.
Afterwards, we will discuss the connection of equation \eqref{obstacle pbm Sec 4} with the notion of semiconcave envelope and we will give the proofs of Theorem \ref{Thm 2nd reach cond} and Corollary \ref{Cor reach cond}. 

\subsection{Proof of Theorem \ref{Thm semiconcave envelope}}

We start by recalling the definition of viscosity solution to problem \eqref{obstacle pbm Sec 4},
in terms of sub- and supersolutions.

\begin{definition}\label{Def vis sol obstacle}
Let $u_T\in \Lip(\R^n)$ be a given function. 
\begin{enumerate}
\item\label{Obstacle subsolution}
The upper semicontinuous function $v$ is a viscosity subsolution of \eqref{obstacle pbm Sec 4}
if for every twice-differentiable function $\phi(\cdot)$,
$$
\phi(x_0)-u_T(x_0)\leq 0 \quad \text{or}  \quad 
-\lambda_n\left[ D^2\phi(x_0) - \dfrac{\left[ H_{pp}(\nabla \phi(x_0)) \right]^{-1}}{T}\right] \leq 0,
$$
whenever $x_0$ is a local maximum of $v-\phi$ and $v(x_0) = \phi(x_0)$.

\item\label{Obstacle supersolution}
The lower semicontinuous function $v$ is a viscosity supersolution of \eqref{obstacle pbm Sec 4}
if for every twice-differentiable function $\phi(\cdot)$,
$$
\phi(x_0)-u_T(x_0) \geq 0 \quad \text{and} \quad 
-\lambda_n\left[ D^2\phi(x_0) - \dfrac{\left[  H_{pp}(\nabla \phi (x_0)) \right]^{-1}}{T}\right]\geq 0,
$$
whenever $x_0$ is a local minimum of $v-\phi$ and $v(x_0)=\phi(x_0)$.

\item\label{Obstacle solution}
The continuous function $v$ is a viscosity solution of \eqref{obstacle pbm Sec 4} if it is at the same time a viscosity super- and subsolution.
\end{enumerate}
\end{definition} 

See \cite{crandall1992user,katzourakis2014introduction} for more details on the theory of viscosity solutions,
and \cite{oberman2007convex} for the particular case of viscosity solutions to degenerate elliptic obstacle problems similar to \eqref{obstacle pbm Sec 4}.

\begin{proof}[Proof of Theroem \ref{Thm semiconcave envelope}]
Since most of the steps in the proof  are in common for the one-dimensional case and 
the quadratic case in any space dimension, we shall treat both cases at the same time, considering a general convex $H$ in any space-dimension.
We will treat both cases separately only when it is needed.
In both cases, the uniqueness of a viscosity solution
can be deduced from Remark \ref{Rmk concave Perron}.

Consider $n\geq 1$, $H$ satisfying \eqref{CondH} and let $L=H^\ast$.
Let us start by proving that $u_T^\ast := S_T^+ (S_T^- u_T)$ is a viscosity supersolution.
Fix $x_0\in\R^n$ and let $\phi$ be a twice-differentiable function such that
$u_T^\ast(x_0)=\phi(x_0)$ and $\phi(x) \leq u_T^\ast(x)$ for any $x\in B_\delta (x_0)$, with $\delta>0$ small.

Since $u_T^\ast(x_0)=\phi(x_0)$, using the inequality in Proposition \ref{Prop forward stability}
it follows that $\phi(x_0)\geq u_T(x_0)$.
By the definition of $u_T^\ast$ and formula \eqref{Hopf formula},
there exists $y^\ast\in\R^n$ such that
\begin{equation}\label{equality sec 4}
\phi(x_0) = u_T^\ast(x_0) = S_T^- u_T(y^\ast) + T\, L\left( \dfrac{x_0-y^\ast}{T} \right).
\end{equation}
In the other hand, using the equality in Proposition \ref{Prop forward stability}, together with formula \eqref{Backward Hopf formula}, we have
$$
S_T^- u_T (y^\ast) = S_T^- u_T^\ast (y^\ast) \geq u_T^\ast (x) - T\, L\left(\dfrac{x-y^\ast}{T} \right),
 \quad \text{for all} \ x\in\R^n,
$$
and this implies
$$
\phi(x) \leq u_T^\ast (x) \leq S_T^- u_T (y^\ast) + T\, L\left(\dfrac{x-y^\ast}{T}\right) \quad \text{for all} \
x\in B_\delta (x_0).
$$
This inequality, combined with \eqref{equality sec 4}, implies that the function
$$
\Psi (x): =  S_T^- u_T (y^\ast) + T\, L \left(\dfrac{x-y^\ast}{T}\right)
$$
satisfies $\phi(x_0)=\Psi(x_0)$ and $\phi(x)\leq \Psi(x)$ for all $x\in B_\delta (x_0)$.

Hence, we deduce that 
$\nabla \phi(x_0) = \nabla \Psi (x_0)$ 
and the Hessian matrix $D^2 (\phi(x_0)-\Psi(x_0))$ is semidefinite negative.

We then have
$$
\nabla \phi(x_0) = L_q\left(\dfrac{x_0-y^\ast}{T}\right), \quad \text{and} \quad
-\lambda_n \left[ D^2 \phi (x_0) - \dfrac{1}{T}L_{qq}\left(\dfrac{x_0-y^\ast}{T}\right)\right] \geq 0.
$$

Now, using Property \ref{Prop Legendre}, we obtain $\dfrac{x_0-y^\ast}{T} = H_p(\nabla \phi(x_0))$,
and then
$$
-\lambda_n \left[ D^2 \phi (x_0) - \dfrac{L_{qq}\left(H_p(\nabla \phi(x_0))\right)}{T}\right] \geq 0.
$$
Using again Property \ref{Prop Legendre}, we conclude
$$
-\lambda_n \left[ D^2 \phi (x_0) - \dfrac{\left[H_{pp}\left(\nabla \phi(x_0)\right)\right]^{-1}}{T}\right] \geq 0.
$$

Next, we prove that $u_T^\ast$ is also a viscosity subsolution.
Fix $x_0\in\R^n$ and let $\phi$ be a twice-differentiable function such that $u_T^\ast(x_0) = \phi(x_0)$
and $\phi(x)\geq u_T^\ast(x)$ for any $x\in B_\delta(x_0)$, with $\delta>0$ small.
We shall prove that 
$$
\text{if $-\lambda_n\left[D^2\phi(x_0) - \dfrac{\left[H_{pp}\left(\nabla \phi(x_0)\right)\right]^{-1}}{T}\right]> 0$, then $\phi(x_0) = u_T(x_0)$.}
$$

Let us suppose that $\lambda_n\left[D^2\phi(x_0) - \dfrac{\left[H_{pp}\left(\nabla \phi(x_0)\right)\right]^{-1}}{T}\right]<0$.
If we take $y_0 = x_0 - T\, H_p\left(\nabla \phi(x_0)\right)$,
using Property \ref{Prop Legendre}, we obtain
$$
\lambda_n \left[ D^2 \phi(x_0) - \dfrac{1}{T}L_{qq} \left(\dfrac{x_0-y_0}{T}\right)\right] <0.
$$
This implies that the function
$$
\psi (x) := \phi(x) - T\, L\left( \dfrac{x-y_0}{T}\right),
$$
is strictly concave in a neighbourhood of $x_0$.
In addition, by the choice of $y_0$,
we can use Property \ref{Prop Legendre} to prove that $\nabla\psi (x_0) =0$.

This implies that $\psi$ has a strict local maximum at $x_0$, 
i.e. there exist $\delta'>0$ small and a constant $C\in\R$, such that
$$
\psi (x_0) = C \quad \text{and} \quad \psi(x)<C \ \text{for all} \ x\in B_{\delta'}(x_0) \setminus \{x_0\}.
$$
Taking $\delta'' = \min \{\delta,\delta'\}$, 
since $\phi\geq u_T^\ast$ in $B_\delta (x_0)$ and $\phi (x_0) = u_T^\ast(x_0)$,
we deduce that the function
$$
\Psi (x) := u_T^\ast(x) - T\, L\left( \dfrac{x-y_0}{T}\right)
$$
also satisfies
\begin{equation}\label{near x0}
\Psi(x_0) = C \quad \text{and} \quad \Psi (x) < C, \ \text{for all} \ 
x\in B_{\delta''}(x_0)\setminus \{x_0\}.
\end{equation}

Now, we shall prove that for the one-dimensional case and for the case when $H$ is quadratic in any space-dimension,
$x_0$ is in fact the unique maximizer of $\Psi$ over $\R^n$.
Here, we treat both cases separately:

\textit{Case 1: One-dimensional case.}
Let us suppose, for a contradiction, that there exists $x_1\in \R\setminus\{x_0\}$ such that $\Psi(x_1)\geq \Psi(x_0)$.
Assume without loss of generality that $x_0<x_1$,
and let
$$x_2 := \underset{x\in [x_0,x_1]}{\text{argmin}} \Psi(x).$$
From \eqref{near x0}, we deduce that $x_2$ is an interior point of $[x_0,x_1]$
and satisfies $\Psi(x_2)<\Psi(x_0)$.

Then, using the definition of subdifferential in \eqref{Def subdiff},
it is not difficult to prove that $0\in D^- \Psi (x_2)$,
and by the definition of $\Psi$, we have
\begin{equation}\label{subdiff x_2}
L_q\left(\dfrac{x_2-y_0}{T}\right) \in D^- u_T^\ast (x_2).
\end{equation}

Now, observe that, by Proposition \ref{Prop semiconcave}, we have that
$u_T^\ast = S_T^+(S_T^- u_T)$ is a semiconcave function.
From a well-known property of semiconcave functions (Proposition 3.3.4 in \cite{cannarsa2004semiconcave}), the superdifferential of a semiconcave function is nonempty for any $x$.

Hence, we have
$$D^- u_T^\ast (x_2)\neq \emptyset \quad \text{and} \quad D^+ u_T^\ast (x_2)\neq \emptyset,$$
and by Proposition 3.1.5 in \cite{cannarsa2004semiconcave}, we deduce that $u_T^\ast$ is differentiable at $x_2$.
Indeed, in view of \eqref{subdiff x_2}, we have
$$
\nabla u_T^\ast (x_2) = L_q\left(\dfrac{x_2-y_0}{T}\right).
$$

Now, by Proposition \ref{Prop unique min Hopf}, 
the unique minimizer in the Hopf-Lax formula \eqref{Hopf formula}, with $u_0 = S_T^- u_T$ and $x=x_2$,
is given by
$$
y^\ast = x_2 - T\, H_p (\nabla u_T^\ast (x_2)) = x_2 - T\, H_p\left( L_q \left(\dfrac{x_2-y_0}{T}\right)\right) = y_0.
$$
Here we used Property \ref{Prop Legendre}.
Therefore, we have
$$
u_T^\ast (x_2) = S_T^- u_T (y_0) + T\, L\left(\dfrac{x_2-y_0}{T}\right).
$$
This implies that $\Psi (x_2)=S_T^- u_T(y_0)$, and using Proposition \ref{Prop forward stability}, we deduce
\begin{eqnarray*}
\Psi (x_2) &=& S_T^- u_T (y_0) = S_T^- u_T^\ast (y_0) \\
&=& \max_{x\in\R} \left\{ u_T^\ast (x) - T\, L\left(\dfrac{x-y_0}{T}\right) \right\} \\
&\geq &  u_T^\ast (x_0) - T\, L\left(\dfrac{x_0-y_0}{T}\right) = \Psi (x_0),
\end{eqnarray*}
and this contradicts $\Psi(x_2) < \Psi(x_0)$. We then deduce that $x_0$ is the unique maximizer of $\Psi$.

\textit{Case 2: $H$ quadratic.}
Suppose that $H$ is given by \eqref{quadratic Hamiltonian} for some positive definite matrix $A$.
As we have proved in the first part of this proof,
$u_T^\ast$ is a viscosity supersolution of \eqref{obstacle pbm Sec 4}.
Since in this case $H_{pp} (p)= A$ for all $p\in\R^n$.
We have that $u_T^\ast$ satisfies the inequality $-\lambda_n \left[D^2 u_T^\ast - \dfrac{A^{-1}}{T}\right] \geq 0$
in a viscosity sense. 

Then, since $L_{qq} (q) = A^{-1}$  for all $q\in\R^n$,
we deduce that $\Psi$ satisfies $-\lambda_n[D^2 \Psi]\geq 0$ in a viscosity sense,
and this implies that $\Psi$ is a concave function (see Theorem \ref{Thm Oberman convex} below).
From the concavity of $\Psi$, together with \eqref{near x0}, we obtain that $x_0$ is the unique maximizer of $\Psi$.

In both cases, we have proved that
\begin{equation}\label{not near x0}
\Psi(x_0) = C \quad \text{and} \quad \Psi (x) <C, \ \text{for all} \ 
x\in \R^n\setminus \{x_0\}.
\end{equation}

Therefore, we have
\begin{equation}\label{not near x0 2 1}
C = u_T^\ast (x_0) - T\, L\left(\dfrac{x_0-y_0}{T}\right)
\end{equation} 
and 
\begin{equation}\label{not near x0 2 2}
C> u_T^\ast (x) - T\, L\left(\dfrac{x-y_0}{T}\right),
\ \text{for all} \ x\in \R^n\setminus \{x_0\},
\end{equation}
and after the equality in Proposition \ref{Prop forward stability} and formula \eqref{Backward Hopf formula}, we obtain
\begin{equation}\label{equality uT uTstar C}
S^-_T u_T (y_0) = S_T^- u_T^\ast (y_0) = C.
\end{equation}
Finally, applying \eqref{not near x0 2 2} and the inequality in Proposition \ref{Prop forward stability}, we obtain
$$
S^-_T u_T(y_0) = C > u_T(x) - T\, L \left( \dfrac{x-y_0}{T}\right), \qquad \text{for all} \ x\in\R^n\setminus\{x_0\},
$$
which, combined with \eqref{equality uT uTstar C}, 
implies $u_T(x_0) - T\, L\left( \dfrac{x_0-y_0}{T}\right) =C$, since the maximum in formula \eqref{Backward Hopf formula} is always attained. 
The equality $\phi(x_0)=u_T^\ast(x_0) = u_T(x_0)$ then follows from \eqref{not near x0 2 1} and the choice of $\phi$.
\end{proof}

\begin{remark}\label{Rmk no general H in multiD}
Observe that in Case 1 of this proof, we used that in the one-dimensional case, 
between two local maxima of a continuous function there is always a local minimum. 
In higher dimension, this property is no longer true since there can be saddle points.
Therefore, the same argument cannot be used to treat the multidimensional case.
However, by assuming that the Hessian matrix of $H$ is constant over $\R^n$,
we can prove that $\Psi$ is concave and then deduce that $x_0$ is the unique maximizer of $\Psi$.
\end{remark}

\subsection{Semiconcave envelopes and proof of Theorem \ref{Thm 2nd reach cond}}\label{Sec semiconcave 1.2}
It is well known that twice-differentiable convex functions are characterized by the property that
the Hessian is everywhere positive semidefinite,
and analogously, twice-differentiable concave functions are those such that
the Hessian is everywhere negative semidefinite.
This result can be generalized to continuous functions by the following result
proved by Oberman in \cite{oberman2007convex}.

\begin{theorem}[Oberman \cite{oberman2007convex}]\label{Thm Oberman convex}
\ 
\begin{enumerate}
\item The continuous function $v:\R^n \to \R$ is convex if and only if it is a viscosity solution of 
$\lambda_1 [D^2 v]\geq 0$.
\item The continuous function $v:\R^n \to \R$ is concave if and only if it is a viscosity solution of 
$\lambda_n [D^2 v]\leq 0$.
\end{enumerate}
\end{theorem}

The definition of viscosity solution for these two inequalities is the same as Definition \ref{Def visc sol diff ineq} below, adapted in an obvious way.
In \cite{oberman2007convex}, it was proved that 
the convex envelope of a given function $f:\R^n \to\R$ is the viscosity solution of the obstacle problem
\begin{equation}\label{convex envelope eq Sec 4}
\max\{ v-f, \, -\lambda_1[D^2 v]\} = 0.
\end{equation}
In \cite{oberman2007convex}, the solution of \eqref{convex envelope eq Sec 4} is also identified with 
the value function of a stochastic optimal control problem.

Analogously to the equation for the convex envelope, the concave envelope of $f$ is the unique viscosity solution of 
$$
\min\{ v-f, \, -\lambda_n[D^2 v]\} = 0.
$$
For related literature concerning equations involving operators like $\lambda_1[\cdot]$ and $\lambda_n[\cdot]$ and its relation 
with convex/concave envelopes and with stochastic control theory
we refer  to \cite{birindelli2018family,blanc2019game,blanc2019games,fleming2006controlled,oberman2011dirichlet,oberman2008computing}
and the references therein.

\begin{figure}[h]
  \centering
\includegraphics[width=6cm, height=4cm]{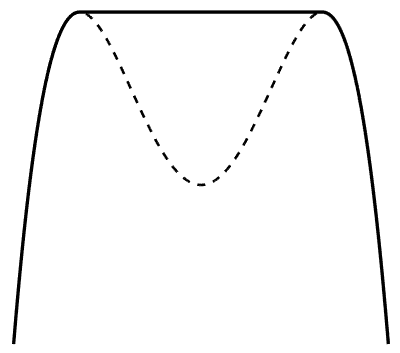}
 \caption{A function $f$, represented by a dotted line, and its concave envelope $f^\ast$.}
\label{FigConcaveEnvelope}
\end{figure}

In our case, the viscosity solution $v$ of the obstacle problem \eqref{obstacle pbm Sec 4} is not necessarily concave.
However, if $H$ is given by \eqref{quadratic Hamiltonian}, the solution of \eqref{obstacle pbm Sec 4} satisfies the inequality
$$
\lambda_n \left[ D^2 v - \dfrac{1}{T}A^{-1}\right] \leq 0,
$$
in the viscosity sense of Definition \ref{Def visc sol diff ineq}.
This implies, using Theorem \ref{Thm Oberman convex}, that the function
$$
x\longmapsto v(x) - \dfrac{\langle A^{-1}x, x\rangle}{2T}
$$
is concave.
Hence, applying the characterization of semiconcave functions in Proposition \ref{Prop semiconcave characterization},
we deduce that $v$ is semiconcave with linear modulus and constant
$$
C = \lambda_n\left[\dfrac{A^{-1}}{T}\right] = \dfrac{1}{T\lambda_1[A]}. 
$$

Next we prove that, when $H$ is quadratic, the solution to the problem \eqref{obstacle pbm Sec 4} can be given
in terms of the concave envelope of a certain auxiliary function $f$.
This result proves the conclusion of Corollary \ref{Cor concave envelope}.

\begin{lemma}\label{Lemma semiconcave concave}
Let $H$ be given by \eqref{quadratic Hamiltonian} for some definite positive $n\times n$ matrix $A$.
Let $T>0$ and $u_T\in \Lip(\R^n)$.
The function $v$ is the viscosity solution of \eqref{obstacle pbm Sec 4} if and only if
the function
$$
w(x) := v(x) - \dfrac{\langle A^{-1}x, x\rangle}{2T} 
$$
is the viscosity solution of 
\begin{equation}\label{concave envelope eq Sec 4}
\min \left\{ w - f, \, -\lambda_n[D^2 w]\right\} =0,
\end{equation}
where
$$
f(x):= u_T (x) - \dfrac{\langle A^{-1}x,x\rangle}{2T}. 
$$
\end{lemma}

\begin{proof}
We recall that the definition of viscosity solution of \eqref{concave envelope eq Sec 4} is the same as
Definition \ref{Def vis sol obstacle} replacing $H$ by $0$ and $u_T$ by $f$.

We need to prove that $v$ is a viscosity supersolution (resp. subsolution) of  \eqref{obstacle pbm Sec 4}
if and only if $w$ is a viscosity supersolution (resp.subsolution) of \eqref{concave envelope eq Sec 4}.
We only give the proof for the equivalence of viscosity supersolution since the same argument works for the viscosity subsolutions.

Let $v$ be a viscosity supersolution of \eqref{obstacle pbm Sec 4}.
For any $x_0\in\R^n$, take $\phi$ any twice-differentiable function such that $w(x_0) = \phi (x_0)$ and 
$x_0$ is a local minimum for $w-\phi$.
Using the definition of $w$ in the enunciate of the lemma, this implies that
\begin{equation*}
v(x_0) - \left( \dfrac{\langle A^{-1}x_0, x_0\rangle}{2T} + \phi(x_0)\right) = 0
\end{equation*}
and
\begin{equation*}
0\leq v(x) - 
\underbrace{\left( \dfrac{\langle A^{-1}x, x\rangle}{2T} + \phi(x)\right)}_{\varphi (x)}, \quad \text{for all} \ x\in B_\delta (x_0),
\end{equation*}
for some $\delta>0$ small.
Observe that $\varphi$ is a twice differentiable function such that $v(x_0) = \varphi(x_0)$
and $x_0$ is a local minimum for $v-\varphi$.
Then, since $v$ is a viscosity supersolution of \eqref{obstacle pbm Sec 4}, we have
$$
\varphi (x_0) \geq u_T(x_0) \quad \text{and} \quad 
-\lambda_n \left[ D^2 \varphi (x_0) - \dfrac{A^{-1}}{T} \right] \geq 0.
$$
Using the definition of $f$ and the choice of $\varphi$, this implies that
$$
\phi (x_0) \geq u_T(x_0) - \dfrac{\langle A^{-1}x_0, x_0\rangle}{2T} = f(x_0)  \quad \text{and} \quad 
-\lambda_n \left[ D^2 \phi (x_0) \right] \geq 0.
$$
We have then proved that $w$ is a viscosity supersolution of \eqref{concave envelope eq Sec 4}.
Similarly, we can prove that if $w$ is a viscosity supersolution of \eqref{concave envelope eq Sec 4},
then $v$ is a viscosity supersolution of \eqref{obstacle pbm Sec 4}.
\end{proof}

We now turn our attention to the reachability condition given in Theorem \ref{Thm 2nd reach cond}:
if $H$ is quadratic or if the space dimension is 1,
then the target $u_T$ is reachable if and only if $u_T$ satisfies the inequality
\begin{equation}\label{diff ineq Sec 4}
\lambda_n \left[ D^2 u_T - \dfrac{\left[ H_{pp} (D u_T)\right]^{-1}}{T}\right] \leq 0
\end{equation}
in the viscosity sense.
Let us make precise the definition of viscosity solution of the above inequality.

\begin{definition}\label{Def visc sol diff ineq}
The upper semicontinuous function $u_T$ is a viscosity solution of
\eqref{diff ineq Sec 4} if for every twice-differentiable function $\phi(\cdot)$,
$$
\lambda_n \left[ D^2\phi (x) - \dfrac{\left[ H_{pp} (\nabla \phi (x)) \right]^{-1}}{T} \right] \leq 0, \quad 
\text{whenever $x$ is a local minimum of $u_T-\phi$}. 
$$
\end{definition}

Observe that it is equivalent to say that $u_T$ is a viscosity supersolution of \eqref{obstacle pbm Sec 4} 
(see \eqref{Obstacle supersolution} in Definition \ref{Def vis sol obstacle}).

\begin{proof}[Proof of Theorem \ref{Thm 2nd reach cond}]
The result can be deduced from Theorems \ref{Thm 1st reach cond} and \ref{Thm semiconcave envelope}.
Let $T>0$ be fixed.
From Theorem \ref{Thm 1st reach cond}, a function $u_T$
satisfies $I_T(u_T)\neq \emptyset$ if and only if 
$$u_T^\ast = S_T^+(S_T^- u_T) = u_T.$$
In view of Theorem \ref{Thm semiconcave envelope}, if $n=1$ or if $H$ is given by  \eqref{quadratic Hamiltonian},
this is equivalent to say that $u_T$ is the viscosity solution of the obstacle problem \eqref{obstacle pbm Sec 4}

Observe that, in view of Definition \ref{Def vis sol obstacle}, $u_T$ is always a viscosity subsolution of \eqref{obstacle pbm Sec 4}.
In order to be also a viscosity supersolution, it only needs to verify the inequality
$$\lambda_n \left[ D^2 u_T (x) - \dfrac{\left[ H_{pp} (D u_T)\right]^{-1}}{T} \right] \leq 0$$
in the viscosity sense of Definition \ref{Def visc sol diff ineq}.
We then deduce that $u_T$ is a viscosity solution of the obstacle problem \eqref{obstacle pbm Sec 4}
if and only if $u_T$ verifies this inequality in the viscosity sense.
\end{proof}

We end this section with the proof of Corollary \ref{Cor reach cond},
that links the reachability of a target $u_T$ with its semiconcavity constant.
At some point in the proof, it will be necessary to use the following elementary property from linear algebra:

\begin{property}\label{Prop eigenvalues}
For any two $n\times n$ symmetric matrices $X$ and $Y$, the following inequality holds:
$$
\lambda_1(X) + \lambda_n(Y) \leq \lambda_n (X+Y) \leq \lambda_n(X) + \lambda_n(Y).
$$
\end{property}

This property can be deduced from the formula
$$
\lambda_1 (X) = \inf_{p\in \R^n} \dfrac{\langle X\, p, p\rangle}{|p|^2}
\quad \text{and} \quad
\lambda_n (X) = \sup_{p\in \R^n} \dfrac{\langle X\, p, p\rangle}{|p|^2}.
$$

\begin{proof}[Proof of Corollary \ref{Cor reach cond}]
Let $u_T$ be a Lipschitz function such that $I_T(u_T)\neq \emptyset$.
From Theorem \ref{Thm 2nd reach cond} we have that $u_T$ is a viscosity solution of
$$
\lambda_n \left[ D^2 u_T - \dfrac{1}{T}A^{-1}\right] \leq 0.
$$
Hence, in view of Theorem \ref{Thm Oberman convex}, the function
$x\mapsto u_T(x) - \dfrac{\langle A^{-1}x, x\rangle}{2T}$
is concave.
Using the characterization of semiconcave functions in Proposition \ref{Prop semiconcave characterization},
we deduce that $u_T$ is semiconcave with linear modulus and constant
$\frac{1}{T\lambda_1 [A]}$.

Let us prove now the second statement.
Suppose that $u_T$ is semiconcave with linear modulus and constant 
$$C = \dfrac{1}{T\lambda_n[A]} = \lambda_1 \left[ \dfrac{A^{-1}}{T}\right].$$

Using the equivalence (i)-(ii) in Proposition \ref{Prop semiconcave characterization},
we have that the function
$$
f(x) := u_T (x) - \dfrac{\lambda_1 [A^{-1}]}{2T} |x|^2
$$
is concave.
Hence, applying Theorem \ref{Thm Oberman convex}, we have that $f$ is a viscosity solution of
$\lambda_n [ D^2 f] \leq 0$.

We then have that $u_T$ satisfies
\begin{equation}\label{cor proof ineq 1}
\lambda_n \left[ D^2 u_T - \dfrac{\lambda_1[A^{-1}]}{T} I_n \right] \leq 0
\end{equation}
in the viscosity sense, where $I_n$ is the $n\times n$ identity matrix.
Now, we combine  \eqref{cor proof ineq 1} with Property \ref{Prop eigenvalues}, 
along with the fact that $\lambda_1[c I_n] = c$, to obtain
\begin{eqnarray*}
0 &\geq & \lambda_n \left[ D^2 u_T - \dfrac{\lambda_1 [A^{-1}]}{T} I_n \right] \\
 &\geq & \lambda_n [D^2 u_T] + \lambda_1 \left[ -\dfrac{\lambda_1[A^{-1}]}{T} I_n\right] \\
 & = & \lambda_n [D^2 u_T] - \dfrac{\lambda_1 [A^{-1}]}{T}.
\end{eqnarray*}
Finally, combining this inequality again with Property \ref{Prop eigenvalues}, we obtain
\begin{eqnarray*}
\lambda_n \left[ D^2 u_T - \dfrac{1}{T} A^{-1} \right] &\leq & 
\lambda_n [D^2 u_T] + \dfrac{\lambda_n[-A^{-1}]}{T} \\
&=& \lambda_n [D^2 u_T] - \dfrac{\lambda_1[A^{-1}]}{T} \leq 0.
\end{eqnarray*}
And from Theorem \ref{Thm 2nd reach cond}, we conclude that $I_T(u_T)\neq \emptyset$.
\end{proof}

\section{Conclusions and perspectives}\label{Sec Conclusions}

The first goal in this work was to characterize the set of reachable targets $u_T$ for a time-evolution Hamilton-Jacobi equation of the form \eqref{HamJac eq}.
Using the inversion of the direction of time in the equation and the notion of backward viscosity solution
we proved, in Theorem \ref{Thm 1st reach cond}, 
that the reachable targets are the fix-points of the composition operator $S_T^+\circ S_T^-$
(see the precise definition of $S_T^+$ and $S_T^-$ in Section \ref{Sec Backward solutions}).
We note that this result can be extended, using similar arguments,
to Hamilton-Jacobi equations with a Hamiltonian depending on $x$, i.e.
\begin{equation*}
\partial_t u + H(x,D_x u) = 0.
\end{equation*}

In addition, for the one-dimensional case
and for quadratic Hamiltonians in any space-dimension
we characterize, in Theorem \ref{Thm 2nd reach cond}, the set of reachable targets as 
the viscosity solutions to the differential inequality
\begin{equation}\label{Reach Ineq Conclusion}
\lambda_n \left[ D^2 u_T - \dfrac{\left( H_{pp} (D u_T)\right)^{-1}}{T} \right] \leq 0.
\end{equation}

Since $H$ is assumed to be strictly convex (see hypothesis \eqref{CondH}), inequality \eqref{Reach Ineq Conclusion}
implies in particular that the reachable targets are semiconcave functions with linear modulus and constant
$$
C = \dfrac{1}{T\, K_{u_T}}, \qquad \text{where} \quad K_{u_T} = \min_{|p|\leq \Lip(u_T)} H_{pp}(p) >0.
$$
Here, $\Lip(u_T)$ is the Lipschitz constant of $u_T$.

Although the semiconcavity of the viscosity solutions is a well-known property of Hamilton-Jacobi equations,
we point out that inequality \eqref{Reach Ineq Conclusion} gives a necessary and sufficient condition
for the reachability of the target.
Inequality \eqref{Reach Ineq Conclusion} describes exactly the semiconcavity that is needed for a target to be reachable.
In one space dimension, inequality \eqref{Reach Ineq Conclusion} can be written as 
$$\partial_{xx} u_T - \left(T\, H_{pp} (\partial_x u_T)\right)^{-1} \leq 0,$$
and is the analogous to the one-sided Lipschitz condition for scalar conservation laws 
(see for example \cite{colombo2019initial,liard2019inverse}).
We recall that the transformation
\begin{equation*}
u (t,x)\longmapsto v(t,x):=\partial_x u (t,x),
\end{equation*}
transforms any viscosity solution of the Hamilton-Jacobi equation \eqref{HamJac eq} in one-space dimension, into an entropy solution of the scalar conservation law
$$
\partial_t v + \partial_x (H(v)) = 0.
$$

In higher dimension, the hypothesis of $H_{pp}(p)$ being constant seems to be crucial in our proof, 
and we do not know if the inequality \eqref{Reach Ineq Conclusion} also provides a necessary and sufficient reachability condition for general strictly convex Hamiltonians in dimension higher than 1.
For the case of a Hamiltonian $H(x,p)$ depending on $x$,
the existence of an inequality similar to \eqref{Reach Ineq Conclusion} 
characterizing the set of reachable targets is a probably difficult problem.

Our second goal in this work was to construct, for any reachable target $u_T$,
the set of all the initial conditions $u_0$ satisfying $S_T^+ u_0 =u_T$.
In Theorem \ref{Thm IniData Identif}, we proved that this construction can be carried out by obtaining two elements:
\begin{enumerate}
\item the function $\tilde{u}_0 := S_T^- u_T$, which in fact happens to be the smallest initial condition
satisfying $S_T^+ u_0 = u_T$;
\item and the set $X_T(u_T)\subset \R^n$, obtained from the differentiability points of $u_T$, where all the initial conditions $u_0$ satisfying $S_T^+ u_0 = u_T$ coincide.
\end{enumerate}
A similar construction was done recently in \cite{colombo2019initial} for the one-dimensional case,
using the equivalence with scalar conservation laws.
However, our approach is based only on the Hamilton-Jacobi setting and applies to any space-dimension.
We stress that we also give, in Theorem \ref{Thm X_I charac} and Proposition \ref{Prop X_I charac general}, 
a rather geometrical identification of the points in $X_T(u_T)$,
which allows to construct the set $X_T(u_T)$ without the necessity of knowing the differentiability points of $u_T$.

The problem of identifying initial sources has been addressed in the context of many different time-evolution models,
and has a relevant interest from both a theoretical and a practical view-point. 
We refer to \cite{casas2015sparse,li2014heat} for works on initial data identification for parabolic equations,
where the fast decay of the solution and the regularizing effect make difficult the reconstruction of the solution at the initial time, specially if the time horizon $T$ is large.
In this setting, the usual approach is to approximate the final target
by considering initial sources constituted by a finite number of Dirac deltas,
where the location of the Dirac deltas as well as the weight for each of them must be carefully determined.

Finally, our last main result concerns the composition operator $S_T^+\circ S_T^-$, which applied to any target $u_T$ (reachable or not), represents a projection of $u_T$ on the set of reachable targets.
For the one-dimensional case and for the case of a quadratic Hamiltonian in any space-dimension
we prove, in Theorem \ref{Thm semiconcave envelope}, that the function $u_T^\ast := S_T^+(S_T^- u_T)$ is the unique viscosity solution to the fully nonlinear obstacle problem
\begin{equation}\label{Obstacle Conclusion}
\min \left\{ v - u_T, \, 
- \lambda_n \left[ D^2 v- \frac{\left[ H_{pp} (D v) \right]^{-1}}{T}\right] \right\} = 0.
\end{equation}
In analogy with the notion of concave envelope of a function, we call the solution to \eqref{Obstacle Conclusion}
the semiconcave envelope of $u_T$.
Note that, if we make $T$ go to infinity, equation \eqref{Obstacle Conclusion} can be viewed as an approximation of the equation for the concave envelope \eqref{concave envelope eq}.

For a fixed $T>0$, the projection $u_T^\ast$ is in fact the smallest reachable target bounded from below by $u_T$.
For the case of a general strictly convex Hamiltonian in any space-dimension, we have not been able to prove the validity of Theorem \ref{Thm semiconcave envelope},
and the existence of an equation similar to \eqref{Obstacle Conclusion} for the case of a more general Hamiltonian $H(x,p)$ depending on $x$ is unknown up to the best of our knowledge.
We would like to end this discussion by noting that projections like $S_T^+\circ S_T^-$, obtained by a backward-forward resolution of the equation, might be of interest in time-evolution problems different from Hamilton-Jacobi equations,
and its identification with the solution of an equation like \eqref{Obstacle Conclusion} could be useful to better understand the nature and properties of this projection.

\bibliographystyle{abbrv}
\bibliography{mybibfile-HJ}

\end{document}